\documentclass[letterpaper, 10 pt, conference]{ieeeconf}  

\IEEEoverridecommandlockouts                              

\usepackage{cite}
\usepackage{amsmath,amsthm, amssymb,amsfonts}
\usepackage{algorithmic}
\usepackage{graphicx}
\usepackage{textcomp}
\def\BibTeX{{\rm B\kern-.05em{\sc i\kern-.025em b}\kern-.08em
    T\kern-.1667em\lower.7ex\hbox{E}\kern-.125emX}}

\usepackage{cleveref}
\usepackage{nicefrac}       
\usepackage{microtype}      
\usepackage{xcolor}  
\usepackage{mathtools}
 \usepackage{algorithm}
\usepackage{caption}
\newtheorem{prop}{Proposition}
\newtheorem{theorem}{Theorem}
\newtheorem{lemma}{Lemma}
\newtheorem{coro}{Corollary}
\newtheorem{assump}{Assumption}
\newtheorem{definition}{Definition}
\newtheorem{example}{Example}

\theoremstyle{remark}
\newtheorem{remark}{Remark}

\newcommand{\dist}{\textup{dist}}
\newcommand{\lmin}{\lambda_{\min}}
\newcommand{\lmax}{\lambda_{\max}}

\newcommand{\tr}{\textup{trace}}
\newcommand{\bR}{\mathbb{R}}

\newcommand{\diag}{\textup{diag}}

\newcommand{\MK}{M^{(H)}}
\newcommand{\HK}{J^{(H)}}
\newcommand{\LK}{L^{(H)}}
\newcommand{\LambdaK}{\mathbf\Lambda^{(H+1)}}
\newcommand{\BK}{\mathbf{B}^{(H)}}
\newcommand{\RK}{\mathbf{R}^{(H)}}
\newcommand{\QK}{\mathbf{Q}^{(H)}}
\newcommand{\Ftrunc}{K_{\textup{trunc}}^{\kappa}}
\newcommand{\AK}{{\bf A}_k^{(H+1)}}
\newcommand{\XiK}{\mathbf{\Xi}^{(H)}}
\newcommand{\SK}{\mathbf{S}^{(H)}}
\newcommand{\gammasys}{\gamma_{\textup{sys}}}

\newenvironment{talign*}
 {\let\displaystyle\textstyle\csname align*\endcsname}
 {\endalign}
\allowdisplaybreaks
\title{On the Optimal Control of Network LQR with Spatially-Exponential Decaying Structure
}
\author{Runyu (Cathy) Zhang, Weiyu Li, Na Li 
\thanks{The work is
supported by NSF AI institute: 2112085, ONR YIP: N00014-19-1-2217, NSF CNS: 2003111. The authors are with the John A. Paulson School of Engineering and Applied Sciences, Harvard University (e-mail: runyuzhang@fas.harvard.edu, weiyuli@g.harvard.edu, nali@seas.harvard.edu). (Corresponding author Runyu (Cathy) Zhang)}}

\allowdisplaybreaks
\begin{document}

\maketitle

\begin{abstract}%
This paper studies network LQR problems with system matrices being spatially-exponential decaying (SED) between nodes in the network. The major objective is to study whether the optimal controller also enjoys a SED structure, which is an appealing property for ensuring the optimality of decentralized control over the network.  We start with studying the open-loop asymptotically stable system and show that the optimal LQR state feedback gain $K$ is `quasi'-SED in this setting, i.e. $\|[K]_{ij}\|\sim O\left(e^{-\frac{c}{\textup{poly}\ln(N)}\dist(i,j)}\right)$. The decaying rate $c$ depends on the decaying rate and norms of system matrices and the open-loop exponential stability constants. Then the result is further generalized to unstable systems under a stabilizability assumption (Assumption \ref{assump:stabilizability}). Building upon the `quasi'-SED result on $K$, we give an upper-bound on the performance of $\kappa$-truncated local controllers, suggesting that distributed controllers can achieve near-optimal performance for SED systems. We develop these results via studying the structure of another type of controller, \emph{disturbance response control}, which has been studied and used in recent online control literature; thus as a side result, we also prove the `quasi'-SED property of the optimal disturbance response control, which serves as a contribution on its own merit.

\end{abstract}

\section{Introduction}
Multi-agent systems has been actively studied in recent years. In many real-life applications such as robotic swarms \cite{liu2018}, power grids \cite{pipattanasomporn2009multi}, smart buildings \cite{zhang16decentralized} etc., the system consists of a group of interactive agents whose dynamics are affected by each other, especially its local neighbors. In general, especially for large scale systems, because of the limited communication among agents, agents needs to take control actions only  based on local observations, motivating the study of distributed learning and control synthesis \cite{witsenhausen1968counterexample, saeks1979decentralized, d2003distributed,rotkowitz2005characterization, cogill2006approximate, rantzer2011distributed,tanaka2011bounded, feng2019exponential, li2021distributed}, where it seeks to find optimal distributed controllers that respect agents' local information constraints.




However, although for certain types of systems (such as quadratic invariant system \cite{rotkowitz2005characterization}) there exist efficient algorithms to find such optimal distributed controllers, not much is known about how the optimum within this information-constrained controller subset differs from the global optimal controller that could access to global information; that is, how much optimality we sacrifice by constraining the controller to be distributed. Answering this important question requires more detailed understanding of the structure of the global optimal controller. If the information pattern of the global optimal controller is distributed or close to distributed,  then it can be expected that the optimal distributed controller could achieve a near-optimal global performance as the global optimal controller. 


The work by \cite{Bamieh02} first proposes to study this question under the setting of spatially invariant systems, where they leveraged the fact that spatially invariant systems can be decoupled under Fourier transform to show that the optimal linear quadratic regulator (LQR) controller is a convolution operator whose kernel has exponentially decaying structure, suggesting that the impact of far-way agents on the optimal control strategy decays exponentially with distance. However, the spatially invariant assumption is relatively restrictive and only holds for graphs with special patterns such as grid or lattice. Later \cite{motee2008optimal} seeks to generalize the result to a broader class which they define as `spatially distributed' class. The paper aims to show that the optimal LQR control gain also lies in the same spatially distributed class given that the matrices $A,B,Q,R$ of the LQR problem are inside the spatially distributed class. However, there exist counter-examples (see e.g. \cite{curtain2009comments} and Section \ref{sec:preliminaries} of this paper) such that $A,B,Q,R$ are spatially distributed, while the optimal controller is not, suggesting that additional conditions are needed in order for the optimal controller to preserve spatially decaying structure.

 There is another line of literature that is related to characterizing the information pattern of optimal LQR controller. We can view the problem as  a purely linear algebraic problem, where the objective is to show whether the LQR operator ${\sf LQR}(A,B,Q,R)$ that calculates the optimal controller preserves the same information structure of the matrix inputs $A,B,Q,R$. There are many existing works in the field of matrix theory that study similar questions for different matrix operators $f$ \cite{Benzi07,Benzi15,haber2016sparse}, such as matrix exponential, matrix inverse and Lyapunov operators, etc. Typically these works try to understand the entrywise pattern of $f(A)$ given that $A$ is a banded or sparse matrix. As far as we know, existing literature mostly consider analytic $f$ which is a function on a single matrix $A$. It remains unclear how the results can be extended to LQR where the operator is a function on multiple system matrices $A,B,Q,R$.

\noindent\textbf{Our Contributions.}
In this paper we consider the standard infinite-horizon discrete time network LQR problem, where there are $N$ agents in the network and each agents has its own local state and local control actions. We focus on the case where $A,B,Q,R$ are spatially exponential decaying (SED, Definition \ref{def:SED}) and $A$ is exponentially stable, and show that the optimal LQR state feedback gain $K$ is `quasi'-SED in this setting, i.e. $\|[K]_{ij}\|\sim O\left(e^{-\frac{c}{\textup{poly}\ln(N)}\dist(i,j)}\right)$ (Theorem \ref{theorem:spatial-decay-F}). The rate $c$ is written out explicitly in terms of the SED rate and norms of system parameters $A,B,Q,R$ and the exponential stability of $A$. Our result can also be extended to the case where matrix $A$ is unstable under the SED stabilizable assumption (Assumption \ref{assump:stabilizability}). As far as we know, we are the first that give concrete decaying rate analysis on the optimal controller for LQR problem with spatially decaying structure. Our result not only gives an answer to the theoretical question of whether spatially distributed LQR also obtains a spatially distributed optimal controller, but also sheds light on how the decaying rates depends on different factors, which provides insights in real applications on how to design the pattern of the information constraints so that the distributed controller can approximate the global optimal as much as possible. 

Building upon the `quasi'-SED result on $K$, we give an upper-bound on the performance of $\kappa$-truncated local  controllers, suggesting that for SED systems, agents can achieve $\epsilon$-optimal performance by only knowing the state information of their $O(\textup{poly}\ln(N)\ln(1/\epsilon))$-neighbors.

Additionally, our proof approaches the problem via \emph{disturbance response parameterization}, thus as a side result, we also prove the `quasi'-SED property of the optimal disturbance response controller. Given that disturbance response controller has its own advantage compared with state feedback controller and is gaining more and more attention in the online learning and control community \cite{simchowitz2020improper,li2021safe,agarwal2019online,agarwal2019logarithmic,goulart2006optimization}, we also believe that it is a contribution on its own merit.%

The work that is the most relevant to our paper is one recent Arxiv preprint \cite{shin2022near}, where they proved the spatially exponential decaying property of the optimal controller $K$ for network LQR problem where the matrix $A$ is sparse and $B,Q,R$ are block-diagonal. Our work differs from theirs in the following aspects: (i) the setting considered in this paper is broader compared with \cite{shin2022near} (ii) The proof techniques are very different. Results in \cite{shin2022near} are derived by leveraging KKT condition, while our proofs are mainly based on disturbance response parameterization. As a result, we also characterize the spatial decaying structure for disturbance response controllers, which is not considered in \cite{shin2022near}.


\noindent\textbf{Notations:} Throughout the paper, we use $\|\cdot\|$ to denote the $\ell_2$ norm of a vector as well as the induced $\ell_2$ norm of a matrix. $\lmax(X), \lmin(X)$ denotes the maximum and minimum eigenvalue of a square matrix $X$ respectively.

\section{Problem Settings and Preliminaries}\label{sec:preliminaries}
We consider the infinite-horizon discrete time network linear quadratic regulator (LQR) problem with $N$ agents $[N]= \{1,2,\dots,N\}$ that are embedded in a network. We assume that there's a distance function defined on the $N$-agent network, i.e. $\dist(\cdot,\cdot): [N]\times[N]\rightarrow \mathbb{R}^{\ge 0}$, such that $\dist(i,j) = \dist(j,i)$, and that triangle inequality holds $\dist(i,j)\le \dist(i,k) + \dist(k,j),~\forall i,j,k\in [N]$. For example, if the $N$ agents are embedded on an un-directed graph, then $\dist(i,j)$ can be taken as the graph distance, which is the length of the shortest path from agent $i$ to $j$. Without further explanation,  $\dist(\cdot,\cdot)$ refers specifically to the \textit{graph distance} throughout the paper, but our results also hold for other types of distances such as Euclidean distance. 

At time step $t$, each agent $i$ has its local state $x_t^i\in \bR^{n_x^i}$ and local control action $u_t^i\in \bR^{n_u^i}$. We use $x_t\in \bR^{n_x}, u_t\in \bR^{n_u}$ to denote the joint state and control action of the $N$ agents, i.e., $n_x = \sum_{i=1}^N n_x^i, n_u = \sum_{i=1}^N n_u^i$ and
\begin{align*}
   \textstyle x_{t} = [x_{t}^{1}; x_{t}^{2}; \cdots; x_t^{N}],\quad 
    u_t = [u_t^{1}; u_t^{2\top}; \cdots; u_t^{N}],
\end{align*}
The dynamic of agent $i$'s state $x_t^i$ is governed by the following linear equation
\begin{equation}\label{eq:dynamic-x-i}
    \textstyle x_{t+1}^i= \sum_{j\in [N]}[A]_{ij}x_{t}^j + [B]_{ij}u_t^j.
\end{equation}
Here the submatrix notation $[X]_{ij}$ for a matrix $X$ denotes the submatrix of $X$ where its row indexes correspond to the role indexes of agent $i$ and its column indexes correspond to the indexes of agent $j$.\footnote{By indexes of agent $i$, we mean that, if the total index length is $n_x$, then the indexes of agent $i$ is of range $\big[\sum_{j=1}^{i-1} n_x^i+1, \sum_{j=1}^{i} n_x^i\big]$. And the same definition also applies for total index length of $n_u$.}

Each agent $i$ also has its local stage cost $$\textstyle c_t^i:=\sum_{j\in[N]}x_t^{i\top} [Q]_{ij}x_t^j + u_t^{i\top} [R]_{ij}u_t^j + 2u_t^{i\top} [S]_{ij}x_t^j$$
which is a quadratic function on $x_t$ and $u_t$. The total stage cost is the summation of all local costs, i.e.,
\begin{align*}
    &\textstyle x_t^\top Qx_t + u_t^\top Ru_t + 2 u_t^\top S x_t = \\
    &\textstyle \qquad \sum_{i,j=1}^N x_t^{i\top} [Q]_{ij}x_t^j + u_t^{i\top} [R]_{ij}u_t^j + 2u_t^{i\top} [S]_{ij}x_t^j.
\end{align*}


Using the system matrices $A,B$ and the cost matrices $Q,R,S$, the network LQR problem is same as the classical LQR formulation shown as below,
\begin{equation}\label{eq:LQR}
\begin{split}
 \min_{\{u_t\}_{t=0}^{\infty}}&\lim_{T\to\infty} \frac{1}{T}\mathbb{E}\sum_{t=0}^{T-1}x_t^\top Q x_t + u_t^\top Ru_t + 2 u_t^\top S x_t\\
 s.t.~&~ x_{t+1} = Ax_t + Bu_t + w_t, ~~w_t\sim\mathcal{N}(0,I).
\end{split}
\end{equation}
Throughout the paper, we make the following assumption on the LQR problem.
\begin{assump}\label{assump:LQR-controllability}
$Q\succ 0, ~~R - SQ^{-1}S^\top \succ 0$.
\end{assump}

The solution to problem \eqref{eq:LQR} is known to be solved by a linear state feedback controller
\vspace{-5pt}
\begin{equation*}
    u_t = -Kx_t\vspace{-10pt}
\end{equation*}
where 
\vspace{-5pt}
\begin{equation}\label{eq:F-expr}
    \textstyle K = (R+B^\top P B)^{-1}(B^\top P A+S).
\end{equation}
Here, the cost-to-go matrix $P$ is the solution to the algebraic Ricatti equation
\begin{equation}\label{eq:ricatti-eq}
    \textstyle P \!=\! A^\top P A - (A^\top PB+S^\top)(R + B^\top PB)^{-1}(B^\top PA+S) + Q.
\end{equation}
Without any special structure of the matrices $A,B,Q,R,S$, the controller $K$ is often a ``dense'' matrix in the sense that for each agent $i$, each $[K]_{ij}$ plays an important role, meaning that for each controller $u^i$, it would need the global state information $x_t^1,\cdots,x_t^N$ of the system. In this paper, we would like to study the information pattern of the global optimal $K$ for special classes of system matrices $A,B, Q,R,S$ and hope the corresponding optimal $K$ would have some structure (close to) being distributed/local which, hence, will ensure near-optimal performance of optimal distributed/local controllers.

\subsection{Spatially exponential decaying structure} 
In this paper, the special class of network LQR problem we consider are the problems with matrices $A,B,Q,R,S$ satisfying the following special decaying property. 
\begin{definition}[Spatially exponential decaying (SED)]\label{def:SED}
A matrix $X\in\bR^{n'\times n''}$ ($n',n'' = n_x$ or $n_u$) is $(c, \gamma)$-spatially exponential decaying (SED) if
\begin{equation*}
\textstyle    \|[X]_{ij}\|\le c\cdot e^{-\gamma \dist(i,j)}, ~~\forall~ i,j\in [N].
\end{equation*}
\end{definition}

We make the following assumption. 
\begin{assump}[SED system]\label{assump:exponential-decay}
There exists $\gammasys  >0$ and constant $a,b,q,r,s>0$ such that $A,B,Q,R,S$ are $(a,\gammasys ), (b,\gammasys ), (q,\gammasys ), (r,\gammasys ), (s,\gammasys )$-SED respectively. Here for convenience, we assume without loss of generality that
$a,b,q,r\ge 1$.
\end{assump}

Assumption \ref{assump:exponential-decay} implies that both the system dynamics and the quadratic cost are \textit{decoupled} across agents in the sense that if $\dist(i,j)$ is large, then $x_j, u_j$ have exponentially small effect on the dynamics of $x_i$ and on the cost of $x_i$.

Below we give one system example that satisfy the SED structure. In Section~\ref{subsec:numerical-practical}, we also provide two practical examples with SED structure. More motivation of studying SED structure is also provided in Appendix \ref{apdx:why-SED}.

\begin{example}[Heat equation on the cyclic graph $\mathbb{Z}_N$]\label{example:circle}
Consider the case where the $N$ agents are located on the cyclic graph $\mathbb{Z}_N$, where agent $i$ is only connected with $i-1$ and $i+1$ (modular $N$). The heat equation dynamic is defined as 
\begin{equation*}
     \textstyle x_{t+1}^i = x_t^i + \eta\left(-2x_t^i + x_t^{i+1} + x_t^{i-1}  + b_i u_t^i\right),
\end{equation*}
thus $A = I-\eta L$, where $L$ is the graph Laplacian of $\mathbb{Z}_N$, and $B = \eta \diag(\{b_i\}_{i=1}^N)$ is a diagonal matrix whose diagonal entries are $\eta b_i$'s.
The control objective is to let the state vector goes to zero while keeping the control energy low, which can be modeled as
\begin{equation}\label{eq:heat-discrete}
    \textstyle \min \lim_{T\to\infty}\frac{1}{T} \sum_{t=1}^T\sum_{i=1}^N q_i {x_t^i}^2 + r_i {u_t^i}^2,
\end{equation}
thus $Q, R$ matrices in this setting are $Q = \diag(\{q_i\}_{i=1}^N)$, $R = \diag(\{r_i\}_{i=1}^N)$.
In this example with small $\eta<1$, we can set $\gammasys = \ln(\frac{1}{\eta})$, $a=1$, $b=\max\{\eta\max_i b_i,1\}$, $q=\max\{\max_i q_i,1\}$ and $r=\max\{\max_i r_i,1\}$ in Assumption \ref{assump:exponential-decay}. \qed
\end{example}

As mentioned in the introduction, one major motivation of our work is to answer the important question of whether an SED system that satisfies Assumption \ref{assump:exponential-decay} still obtains the optimal LQR control gain $K$ which is also SED. Unfortunately, this is not necessarily true without additional conditions, as shown in the counter-example below.

\subsection{A Counter-example}
We can construct the following example, where $A,B,Q,R,S\in \mathbb{R}^{N\times N}$ which are all SED, whereas the optimal control gain $K$ is not.
\vspace{-10pt}
\begin{align}\label{eq:counter-example}
    \!\!\!\!A \!=\! 1.1 I, ~Q \!=\!R\! = \!I,~ S \!=\! 0,~
    B \!=\!\! 
    \begin{bmatrix}
         1& 1 &&& \\
         & 1 & 1&&\\
         &  &\ddots&\ddots&\\
         & & & 1 & 1\\
         & & &  & 1
    \end{bmatrix}\!\!.
\end{align}
 Since $A,Q,R$ are scalar matrices, and $B$ is a banded matrix, clearly they are SED matrices with respect to the graph distance defined by $\mathbb{Z}_N$, hence the above system satisfies Assumption \ref{assump:exponential-decay}.
\vspace{-8pt}
\begin{figure}[htbp]
    \centering
    \begin{minipage}{.40\linewidth}
    \includegraphics[width=.9\textwidth]{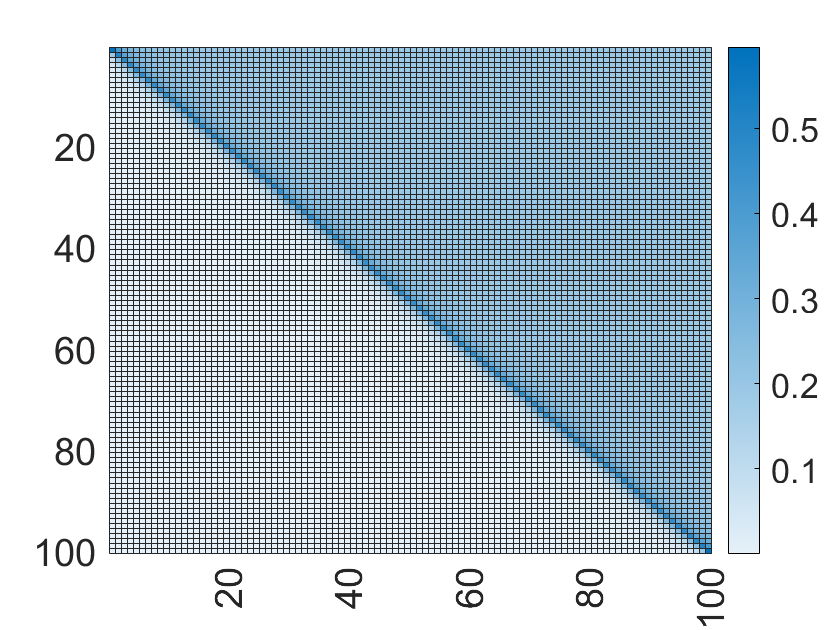}
    \vspace{-5pt}
    \caption{\small Heatmap of $\text{abs}(K)$. An entry is closer to zero if its color is whiter.}
    \label{fig:counter-example-1}
    \end{minipage}
    \hspace{20pt}
    \begin{minipage}{.40\linewidth}
    \includegraphics[width=.9\textwidth]{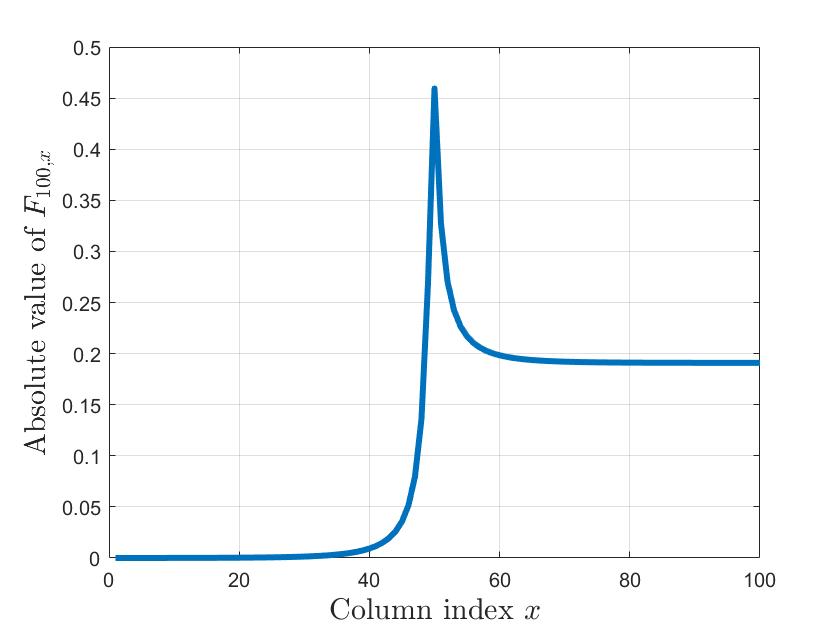}
    \vspace{2pt}
    \caption{\small Entries of the $50$-th row of $\textup{abs}(K)$.}
    \label{fig:counter-example-2}
    \end{minipage}
\end{figure}
We pick $N=100$ and directly call the {\sf dlqr} function in MATLAB to solve for the optimal controller. To visualize the structure of the optimal controller $K$, Figure \ref{fig:counter-example-1} plots the heatmap of the absolute value of entries in $K$; Figure \ref{fig:counter-example-2} plots the absolute value of the entries of a row of $K$. From the two figures we can see that the upper-triangular entries have non-negligible absolute values, implying that $K$ is not SED. Thus this serves as a counter-example showing that the optimal controller $K$ is not SED even if $A,B,Q,R, S$ are SED.

\section{Main Results}\label{sec:main-result}
The previous counter-example suggests that additional condition is needed in order for $K$ to also be SED. In this section, we first look into the open-loop exponential stable system (Assumption \ref{assump:A-exponential-stable}) and then extend our results to unstable systems, but with an additional assumption on SED stabilizability (Assumption \ref{assump:stabilizability}). Lastly, we illustrate how our results ensure a small optimality gap between  the `$\kappa$-truncated local controllers' and the global optimal controller $K$. 

\subsection{For open-loop exponentially stable systems}\label{sec:open-loop-stable}
In the counter-example in Equation~\eqref{eq:counter-example}, if we replace $A = 1.1I$ to $A = 0.9I$, the optimal control gain $K$ immediately becomes a SED matrix (See more details in Section \ref{sec:toy-example}, figure \ref{fig:counter-example-stable-heatmap}). This observation suggests that stability indeed plays an important role in the decaying structure of $K$, and inspire us to first consider the open-loop stable systems. We focus on systems that are asymptotically stable, which is equivalent to exponentially stable for linear systems \cite{anderson2007optimal}. 
To characterize the exact decaying rate of the optimal controller $K$, we provide the following definition of exponential stability. 
\begin{definition}[$(\tau,e^{-\rho})$-stability] For $\tau\ge1, \rho>0$, we define a matrix $A$ as $(\tau, e^{-\rho})$-stable if
    $\|A^k\|\le \tau e^{-\rho k}.$
\end{definition}

\begin{assump}[Open-loop Stability]\label{assump:A-exponential-stable}
The linear system considered in \eqref{eq:LQR} is open-loop asymptotically stable.
\end{assump}

Additionally, for LQR problems with Assumption~\ref{assump:LQR-controllability} and \ref{assump:A-exponential-stable}, the optimal $K$ always asymptotically stabilizes the system (c.f. \cite{anderson2007optimal}). Thus without loss of generality we assume that there exist some $\tau\ge1, \rho>0$ such that both $A$ and $A-BK$ are $(\tau, e^{-\rho})$-stable, i.e.
\begin{equation}\label{eq:open-closed-loop-stable}
    \textstyle \|A^k\|, \|(A-BK)^k\|\le \tau e^{-\rho k}.
\end{equation}

With Assumption \ref{assump:A-exponential-stable}, we are able to prove the following theorem. For compactness, here we only state the main result and briefly discuss the implication. The intuition and proof sketches are provided in the later sections. The full proof can be found in Appendix \ref{apdx:proof-SED-F}.
\begin{theorem}\label{theorem:spatial-decay-F}
Under Assumption \ref{assump:LQR-controllability}, \ref{assump:exponential-decay}, and \ref{assump:A-exponential-stable}, the optimal control gain $K$ for problem \eqref{eq:LQR} is $(c_K,\gamma_K)$-SED, where 
\begin{align*}
    c_K &\lesssim O\bigg(\textup{poly}\Big(N, \tau,\frac{1}{1-e^{-2\rho}}, \frac{1}{\lmin(R\!-\!SQ^{-1}S^\top)},\\
    &\qquad \qquad\qquad\qquad\|B\|,\|Q\|,\|S\|,\|K\|, b, q, s\Big)\bigg), \\
    \gamma_K &\gtrsim  \gammasys \rho\bigg\slash O\bigg(\textup{poly}\Big(\tau,\frac{1}{1-e^{-2\rho}},\frac{1}{\lmin(R\!-\!SQ^{-1}S^\top)},\\
    &\qquad\!\!\lmax(R),\|B\|,\|Q\|,\|S\|,\ln\left\{N,a,b,q,r,s,\gammasys \right\}\Big)\bigg)\!,
\end{align*}
where $\tau,\rho$ are defined as in \eqref{eq:open-closed-loop-stable}.
\end{theorem}

\begin{remark}[Discussion on the SED rate $\gamma_K$]\label{rmk:SED-rate-validness} Since this paper considers the finite dimension problem, any matrix is $(c,\gamma)$-SED as long as $c$ is picked to be arbitrarily large and $\gamma$ arbitrarily close to zero. For example, for the cyclic graph $\mathbb{Z}_n$, if $\gamma_K = 1/N$, then Theorem \ref{theorem:spatial-decay-F} only gives that $\|[K]_{ij}\|$ is upper-bounded by $c_K e^{-\frac{\dist(i,j)}{N}}\ge c_K e^{-1}$, which is a non-negligible constant and thus the bound is invalid and not informative. Therefore in order for the SED rate to be valid, $\gamma_K$ should not `scale with $N$'. Intuitively this means that if we enlarge the dimension of the problem, these factors have upper/lower-bounds that do not increase/decrease accordingly. Hence we would like to emphasize that, Theorem \ref{theorem:spatial-decay-F} needs to be applied with caution. Before concluding that $K$ is SED, it is necessary to check that all factors in Theorem \ref{theorem:spatial-decay-F}, i.e. $\gammasys ,\rho,\tau, \lmin(R),\lmax(R),\|B\|,\|Q\|\dots$, do not scale with $N$. To make this more concrete, we calculate these factors explicitly for a toy example (Example \ref{example:toy}) in the next section.

Additionally, apart from these system factors, there's one term that always scales with $N$, which is the $O(\textup{poly}\ln(N))$ showing up in the denominator, making the bound worse as $N$ grows larger, thus more accurately speaking the decaying rate of the entries $[K]_{ij}$ is `quasi'-SED, i.e., $O\left(\exp\left({-\frac{c\cdot \dist(i,j)}{\textup{poly}\ln(N)}}\right)\right)$, which is slightly worse than exponential decay. It is currently unclear to us whether the $O(\textup{poly}\ln(N))$ factor is fundamental or a proof artifact. It remains future work to answer whether this term could be removed. However, we remark that the `quasi'-exponential bound is good enough because $\ln(N)$ grows very slowly with $N$, so as long as $\dist(i,j)\gtrsim O(N^\epsilon)$ given any fixed constant $\epsilon >0$, we can still get that $\|[K]_{ij}\|$ is close to zero since $\lim_{N\to+\infty} \frac{N^\epsilon}{\textup{poly}\ln(N)} = +\infty$. Thus the bound works for a wide range of systems, such as the cyclic graph $\mathbb{Z}_N$ in Example \ref{example:circle}, or lattice graph etc. More generally speaking, the bound is valid as long as the graph is not well-connected. \qed 
\end{remark}

\subsection{Extension to unstable systems}\label{sec:unstable-system}
Theorem \ref{theorem:spatial-decay-F} can be extended to the case where $A$ is unstable or marginally stable under an additional mild assumption on the stabilizability.
\begin{assump}[SED Stabilizability]\label{assump:stabilizability}
There exists a $K_0 \in \mathbb{R}^{n_u\times n_x}$ that is $(k_0, \gammasys )$-SED such that $A-BK_0$ is $(\tau, e^{-\rho})$-stable.
\end{assump}
\begin{coro}[of Theorem~\ref{theorem:spatial-decay-F}, extension to unstable case]\label{coro:spatial-decay-F-unstable-A}
Under Assumption \ref{assump:LQR-controllability}, \ref{assump:exponential-decay}, and \ref{assump:stabilizability}, and that the system is closed-loop stable, i.e. $A-BK$ is $(\tau,e^{-\rho})$-stable, the optimal control gain $K$ of problem \eqref{eq:LQR} is $(c_K,\gamma_K)$-SED, where 
\begin{align*}
    c_K &\!\lesssim\! O\!\bigg(\!\textup{poly}\Big(\!N\!, \tau\!,\frac{1}{1-e^{-2\rho}}, \frac{1}{\lmin(R\!-\!SQ^{-1}S^\top)},\frac{1}{\lmin(Q)},\\
    &\qquad \quad\|B\|,\|Q\|,\|R\|,\|S\|,\|K\|,\|K_0\|^2, b, q, r, s, k_0\Big)\bigg), \\
    \gamma_K &\!\gtrsim\!  \gammasys \rho\!\!\bigg\slash\!\! O\!\bigg(\!\textup{poly}\Big(\!\tau\!,\frac{1}{1\!-\!e^{\!-\!2\rho}},\!\frac{1}{\lmin\!(R\!-\!SQ^{-1}S^\top)},\!\frac{1}{\lmin\!(Q)},\\
    &\hspace{-12pt}\lmax\!(R)\!,\|B\|,\|Q\|,\|R\|,\|S\|,\|K_{\!0}\|^2\!\!,\ln\!\left\{\!N\!,a\!,b\!,q\!,r\!,s\!,k_0\!,\gammasys \!\right\}\!\!\Big)\!\!\bigg)\!.
\end{align*}
\end{coro}
The proof of Corollary \ref{coro:spatial-decay-F-unstable-A} is deferred to Appendix \ref{apdx:proof-SED-F}. We would also like to note that the counter-example provided in Section \ref{sec:preliminaries} is not in conflict with Corollary \ref{coro:spatial-decay-F-unstable-A}, because we can show that it does not satisfy Assumption \ref{assump:stabilizability} (detailed verification see Appendix \ref{apdx:counter-example}). Here we provide one toy example to justify the usefulness of the result.

\begin{example}[Heat equation revisited]\label{example:toy}
We consider the same cyclic graph $\mathbb{Z}_N$ and the heat equation as in example \ref{example:circle}. For the sake of simplicity, here we set $b_i = q_i = 1, r_i = \alpha$, i.e.,
\vspace{-5pt}
\begin{align*}
&\textstyle \min \lim_{T\to\infty}\frac{1}{T} \sum_{t=1}^T\sum_{i=1}^N  {x_t^i}^2 + \alpha {u_t^i}^2,\\
     &\textstyle x_{t+1}^i =  x_t^i + \eta\left(-2x_t^i + x_t^{i+1} + x_t^{i-1}  + u_t^i\right),
\end{align*}
Thus $A = I-\eta L, B =\eta I, Q = I, R = \alpha I, S=0$, where $L$ is graph Laplacian of  $\mathbb{Z}_N$. From now on we assume that $\eta$ is small enough such that $A$ is a positive-semi-definite matrix, i.e., $\eta \le \frac{1}{4}$.

Note that $A$ is not exponentially stable as it has $1$ as its eigenvalue. Thus Theorem \ref{theorem:spatial-decay-F} is not directly applicable. Yet we could set $K_0 = I$, and $A-BK_0 = (1-\eta) I-\eta L \Longrightarrow \|A-BK_0\| \le e^{-\eta}$, i.e., $A-BK_0$ is $(1,e^{-\eta})$-stable, hence Assumption \ref{assump:stabilizability} is satisfied and Corollary \ref{coro:spatial-decay-F-unstable-A} can be applied.

When applying Corollary \ref{coro:spatial-decay-F-unstable-A}, we can set $\gammasys  = \ln(\frac{1}{\eta}), a=1, b=1, q=1, r=\max\{\alpha,1\}$ in Assumption \ref{assump:exponential-decay}, and set $\tau = 1$, $\rho = \eta$, $\|K_0\| = 1$, $k_0 = 1$ in Assumption \ref{assump:stabilizability}. Additionally, the system norms are bounded by $\|B\|=\eta, \|Q\|=1, \lmax(R) = \lmin(R)=\alpha$. Note that all these terms are constants that does not depend on the number of agents $N$, thus for any $\mathbb{Z}_N$, $\gamma_K$ in Theorem \ref{theorem:spatial-decay-F} is on the scale $\gamma_K \gtrsim \frac{c}{\textup{poly}\ln(N)}$, where $c$ is some constant that does not depend on $N$, which implies that the optimal $K$ for this example is indeed quasi-SED.\qed

\end{example}

\subsection{Performance of the truncated local controller}\label{sec:performance-truncated-controller}
The above results shed light on how well can distributed controllers approximate the global optimal performance for SED systems. This section looks into the performance of the $\kappa$-truncated local controller defined as follows:
\begin{equation*}
    [\Ftrunc]_{ij} = \left\{
    \begin{array}{ll}
        [K]_{ij} & \textup{if } \dist(i,j)\le \kappa-1 \\
         0& \textup{otherwise} 
    \end{array}
    \right..
\end{equation*}
The $\kappa$-truncated local controller is desirable for distributed control because each node only requires the state information from its $\kappa$-hop neighborhood to calculate its own control actions.
We measure the performance in terms of the LQR cost, i.e., for any exponentially stable controller $K'$, we can define its cost $C(K')$ as:
\vspace{-5pt}
\begin{equation*}
\begin{split}
C(K') :=\min_{\{u_t\}_{t=0}^{\infty}}&\lim_{T\to\infty} \frac{1}{T}\mathbb{E}\sum_{t=0}^{T-1}x_t^\top Q x_t + u_t^\top Ru_t + 2u_t^\top Sx_t\\
 s.t.~~ x_{t+1} &= Ax_t + Bu_t + w_t, ~~w_t\sim\mathcal{N}(0,I),\\
u_t &= K'x_t
\end{split}
\end{equation*}

\begin{theorem}[Appendix \ref{apdx:performance-truncated}]\label{theorem:performance-truncated} 
Assume that the optimal control gain $K$ is $(c_K, \gamma_K)$-SED, then for $\kappa \ge \frac{\ln\left(\frac{2\tau c_K\sqrt{N}\|B\|}{1-e^{-\rho}}\right)}{\gamma_K},$
we have that
\begin{equation*}
\begin{split}
&\textstyle\quad C(\Ftrunc) - C(K) \\
&\textstyle \le \frac{2\tau}{1-e^{-\rho}} \|R+B^\top P B\|\sqrt{N\min\{n_x,n_u\}} c_K e^{-\gamma_K\kappa},
\end{split}
\end{equation*}
where $P$ is defined in \eqref{eq:ricatti-eq}.
\end{theorem}
Theorem \ref{theorem:performance-truncated} suggests that for SED systems that satisfies Assumption \ref{assump:A-exponential-stable} or \ref{assump:stabilizability}, in order to achieve $\epsilon$-optimal performance in terms of the LQR cost, we only need to set $\kappa \sim \textup{poly}\ln(N)\ln(1/\epsilon)$, i.e., each agent only need to know the information of its $\kappa$-hop neighbors, which is only a negligible proportion of the total number of agents (for graphs that are not well-connected), to achieve near-optimal performance.




\section{Proof Enabler: Disturbance Response Controller}\label{sec:disturbance-response}


The rest of the paper mainly focuses on providing insights and proof sketches of the main results, especially  Theorem \ref{theorem:spatial-decay-F}. One of the major technical difficulties is that the algebraic Ricatti equation \eqref{eq:ricatti-eq} is a \textit{nonlinear} matrix equation, and it is hard to verify how the spatially decaying structure is preserved given that the matrix coefficients $A,B,Q,R,S$ in \eqref{eq:ricatti-eq} are SED. This observation inspires us to think about the problem from a different perspective --- doing convex reparameterization of \eqref{eq:LQR} using disturbance response. In this section, we consider the following finite-truncated disturbance response controller optimization problem of \eqref{eq:LQR}
\begin{align}
    \textstyle \min_{\LK}\quad  & \textstyle C\left({\textstyle \LK:=\left[{\LK_1}; \cdots; {\LK_H}\right]}\right)\notag\\
    &\textstyle \hspace{-40pt}= \lim_{T\to\infty} \frac{1}{T}\mathbb{E}\sum_{t=0}^{T-1}x_t^\top Q x_t + u_t^\top Ru_t + 2u_t^\top Sx_t\label{eq:LQR-disturbance-feedback}\\
    \textstyle s.t.~~ x_{t+1} & \textstyle = Ax_t + Bu_t + w_t, ~~w_t
    \sim\mathcal{N}(0,I)\notag\\
    \textstyle u_t = &\textstyle  L_1^{(H)} w_{t-1} + \cdots+L_H^{(H)}w_{t-H}~~(w_s \!=\! 0 \textup{ for } s \!<\!0).\notag
\end{align}
This problem is a special case of Youla parameterization \cite{youla1976modern} for the open-loop stable LQR case. The reason that we consider problem \eqref{eq:LQR-disturbance-feedback} is twofold: (i) There is a simple relationship of the solution of \eqref{eq:LQR-disturbance-feedback} with the optimal control gain $K$ (formally stated in Section \ref{sec:relationship-K-L}); (ii) Compared with solving the Ricatti equation for $K$, \eqref{eq:LQR-disturbance-feedback} is a quadratic unconstrained optimization problem w.r.t. $L_1^{(H)},\dots,L_H^{(H)}$, so it can be expected that the problem obtains an explicit solution by solving a system of linear equations, whose SED structure is easier to analyze.
\subsection{Relationship with the optimal control gain $K$}\label{sec:relationship-K-L}
\begin{lemma}[\cite{supp2}]\label{lemma:relationship-K-L}
Let $K$ be the optimal control gain of the LQR problem \eqref{eq:LQR}, then under Assumption \ref{assump:A-exponential-stable}, $\LK$ solved from \eqref{eq:L-tilde-eq} satisfies
\begin{equation}
    \|K \!+\! L_1^{(H)}\| \le \frac{2\tau^3(\|B\|^2\|K\|\|Q\|+\|B\|\|K\|\|S\|)}{\lmin(R\!-\!SQ^{-1}S^\top)(1-e^{-2\rho})^{5/2}}e^{-H\rho}.
\end{equation}
\end{lemma}
Lemma \ref{lemma:relationship-K-L} suggests that, rather surprisingly, $-\LK_1$ is a good approximation of $K$, and the approximation error decays exponentially with the truncation horizon $H$. As a consequence, in order to understand the SED structure for $K$, we could instead study the SED structure for $\LK_1$, which is essentially easier.

\subsection{SED structure of $\LK$}
We first state the main theorem of this section:
\begin{theorem}[SED structure of $\LK$]\label{theorem:spatially-decay-L}
Under Assumption \ref{assump:exponential-decay} and \ref{assump:A-exponential-stable}, the optimal $\LK$ that solves \eqref{eq:LQR-disturbance-feedback} satisfies that $\LK_k$ is $(c_L^{(H)}, \gamma_L^{(H)})$-SED for any $1\le k\le K$, where 
\begin{align*}
    &c_L^{(H)} \lesssim O\bigg(\textup{poly}\Big(N, \tau,\frac{1}{1-e^{-2\rho}}, \frac{1}{\lmin(R\!-\!SQ^{-1}S^\top)},\\
    &\qquad \qquad\qquad\qquad \|B\|,\|Q\|,\|S\|, b, q, s\Big)\bigg), \\
    &\gamma_L^{(H)} \gtrsim \gammasys \rho{\bigg\slash}O\bigg(\textup{poly}\Big(\tau,\frac{1}{1-e^{-2\rho}},\frac{1}{\lmin(R\!-\!SQ^{-1}S^\top)},\\
    &\qquad \qquad \lmax(R),\|B\|,\|Q\|,\|S\|,\ln\left\{N,H,a,b,q,r,s\right\}\Big)\bigg)\!.
\end{align*}
\end{theorem}
The full proof as well as formal statement of Theorem \ref{theorem:spatially-decay-L} is in Appendix \ref{apdx:spatially-decay-L}. Along with Lemma \ref{lemma:relationship-K-L}, the SED structure of $\LK$ in Theorem \ref{theorem:spatially-decay-L} can be easily translated to the optimal control gain $K$, and thus obtain Theorem \ref{theorem:spatial-decay-F} (detailed derivation see Appendix \ref{apdx:proof-SED-F}).

We also view Theorem \ref{theorem:spatially-decay-L} itself as an interesting finding. There are many recent works that use disturbance response controllers for online adaptive control \cite{simchowitz2020improper,li2021safe,agarwal2019online,agarwal2019logarithmic,goulart2006optimization}. It would be an exciting direction to see how the SED structure of the optimal disturbance response controller could lead to more efficient learning/control algorithms in these problems.

\section{Proof Sketches}\label{sec:MK-HK-LK-structure}
The previous section has introduced the relationship of disturbance response controller and the state feedback controller (Lemma \ref{lemma:relationship-K-L}) and shown the SED structure of the disturbance response controller (Theorem \ref{theorem:spatially-decay-L}). Combining these two insights finishes the proof our our main result Theorem \ref{theorem:spatial-decay-F}. Thus this section focuses on a detailed proof sketch of Theorem \ref{theorem:spatially-decay-L}, which is the key step in proving the main results. We start with Section \ref{sec:explicit-solu-disturbance} which shows that the explicit solution for the optimization problem \eqref{eq:LQR-disturbance-feedback} is a solution to a system of linear equations (Lemma \ref{lemma:optimal-L-K}, Equation \eqref{eq:L-tilde-eq}). Then Section \ref{sec:M-J-structure} studies the SED structure for matrices coefficients of the system of linear equations. Building on these results, Section \ref{sec:structure-L} gives a detailed proof sketch of Theorem \ref{theorem:spatially-decay-L}.

\subsection{Solution of problem \eqref{eq:LQR-disturbance-feedback}}\label{sec:explicit-solu-disturbance} 
Note that \eqref{eq:LQR-disturbance-feedback} is a quadratic unconstrained optimization problem w.r.t. $\LK$, so it can be expected that the problem obtains an explicit solution. Since our objective is to show that the solution $\LK_1,\dots,\LK_H$ to \eqref{eq:LQR-disturbance-feedback} are SED. As an initial step, it is necessary to identify the explicit solution, which is the main focus of this section. We first state the main lemma:
\begin{lemma}[Appendix \ref{apdx:optimal-LK}]\label{lemma:optimal-L-K}
The optimal $\LK$ of \eqref{eq:LQR-disturbance-feedback} solves
\begin{equation}\label{eq:L-tilde-eq} 
    \MK\LK + \HK=0,
\end{equation}
where $\MK\in\bR^{Kn_u\times Kn_u}$ and $\HK\in\bR^{Kn_u\times n_x}$ are defined as
\begin{align}
    \hspace{-10pt}\MK\! \!:=\! \begin{bmatrix}
         M_{11}& M_{12}&\cdots &M_{1H} \\
         M_{21}& M_{22}&\cdots &M_{2H}\\
         \vdots&\vdots & &\\
         M_{H1} & M_{H2}&\cdots &M_{HH}
    \end{bmatrix}
    \HK := \begin{bmatrix}
    J_1\\
    \vdots\\
    J_H
    \end{bmatrix}, \label{eq:def-MK-HK}
\end{align}
with submatrices $ M_{km}\in \bR^{n_u\times n_u}, J_k \in \bR^{n_u\times n_x}$ defined as:
\begin{align}
    M_{km}&:=\begin{cases}
        B^\top GB + R, &  k=m\\
        B^\top GA^{k-m}B + SA^{k-m-1}B, &  k> m\\
        B^\top(A^{m-k})^\top GB + B^\top (A^{m-k-1})^\top S^\top, & k<m 
    \end{cases},\\
  J_k&:= B^\top G A^k + SA^{k-1}. \label{eq:def-M-H}
\end{align}
Here $G\in R^{n_x\times n_x}$ is defined as:
\begin{align}\label{eq:def-G}
\textstyle
    G := \sum_{t=0}^\infty (A^t)^\top Q A^t.
\end{align}
\end{lemma}
\vspace{10pt}
Although the definition of variables may seem heavy at first glance, the key takeaways of Lemma \ref{lemma:optimal-L-K} are actually simple and straightforward: (i) The optimal $\LK$ solves a system of linear equations (Eq \eqref{eq:L-tilde-eq}); (ii) The submatrices $M_{km}, J_k$ of $\MK,\HK$ are matrix polynomials with respect to the LQR parameters $A,B,Q,R,S$ as well as $G$ (which is the solution of the discrete time Lyapunov equation \eqref{eq:def-G}). These two observations are the key enablers for analyzing the SED structures of $\LK$, which is explained in more detail in the following subsections.

\subsection{SED structure of $M_{km}, J_k$'s}\label{sec:M-J-structure}
\begin{lemma}[SED structure of $M_{km}, J_k$'s]\label{lemma:spatially-decay-M-H}
Under Assumption \ref{assump:exponential-decay} and \ref{assump:A-exponential-stable}, for any $k,m\ge 1$, $M_{km}$ is $(c_M, \gamma_M)$-SED and $J_k$ is $(c_J,\gamma_M)$-SED ($M_{km}, J_k$ defined in \eqref{eq:def-M-H}), where the absolute constants are defined as
\begin{align*}
\textstyle c_M &= b^2N^2\left(\frac{\tau^2\|Q\|}{1-e^{-2\rho}}+2q\right) + bN(s+\tau\|S\|)+ r,\\
\textstyle c_J &= bN\left(\frac{\tau^2\|Q\|}{1-e^{-2\rho}} + 2q\right) + s + \tau\|S\|,\\
\textstyle \gamma_M &= \frac{\gammasys \rho}{(\rho+\ln(aN))}.
\end{align*}
\end{lemma}
The proof of Lemma \ref{lemma:spatially-decay-M-H} is deferred to Appendix \ref{apdx:spatially-decay-M-H}. Here we provide a brief intuition of the lemma. The key fact is that $G = \sum_{t=0}^\infty (A^\top)^t Q A^t$, which shows up in every $M_{km}$ and $J_k$. Although as far as we know we are not aware whether exact same results exist in literature, the proof technique resembles previous works that study the setting where $A,Q$ are sparse or banded matrices \cite{haber2016sparse}. Then, using similar arguments, we can show that $M_{km}, J_k$'s are also SED.

\subsection{SED structure of $\LK$}\label{sec:structure-L}
 
Now we are ready to present a proof sketch of Theorem \ref{theorem:spatially-decay-L} which also provides important insights on how the SED rate $\gamma_L^{(H)}$ depends on other factors. The proof sketch can roughly be decomposed into the following four steps: 
\paragraph{Step 1: Taylor expansion of $\left(\MK\right)^{-1}$} It is not hard to show that $\MK$ is a positive definite matrix with bounded eigenvalues, namely, $0\prec \lmin I \preceq \MK\preceq \lmax I$ for some $\lmin,\lmax$. Normalize $\MK$ as ${\MK}' := I - \frac{\MK}{\lmax}$, then we have that $0\preceq {\MK}' \preceq (1-\frac{\lmin}{\lmax}) I$.
Applying Taylor expansion to Lemma \ref{lemma:optimal-L-K}, we can write $\LK$ can as:
\begin{align*}
    &\textstyle \LK \!=\! -\!\left(\MK\right)^{\!-1}\!\!\HK \!=\!-\frac{1}{\lmax}\left(I-{\MK}'\right)^{-1}\!\!\HK\\
    &\textstyle ~=\!-\frac{1}{\lmax}\sum_{s=0}^{+\infty}{{\MK}'}^s \!\HK.
\end{align*}
Define the truncated Taylor series as $$\textstyle L^{(H),t} := -\frac{1}{\lmax}\sum_{s=0}^{t-1}{{\MK}'}^s \HK.$$ Note that $L^{(H),t}$ is a $(t-1)$-th order polynomial w.r.t. $\MK$, and it converges to $\LK$ as $t\to +\infty$.
\paragraph{Step 2: $L^{(H),t}$'s (do not) preserve SED property.} Since $L^{(H),t}$ converges to $\LK$, the natural next-step is to check whether $L^{(H),t}$ has SED property. From its definition, $L^{(H),t}$ satisfies the following linear recursive relation:
\begin{align*}
   \textstyle L^{(H),t+1} &=  \textstyle \left(I-\frac{\MK}{\lmax}\right)L^{(H),t}- \frac{\HK}{\lmax}\\
    \textstyle \Longrightarrow~~L^{(H),t+1}_k &=\textstyle  -\sum_{m=1}^H\frac{M_{km}}{\lmax}L^{(H),t}_m + L^{(H),t}_k - \frac{\HK_k}{\lmax}.
\end{align*}
Additionally, we can prove the following lemma suggesting that SED property is preserved for matrix addition and multiplication:
\begin{lemma}[Appendix \ref{apdx:auxiliary}] \label{lemma:bound-AB-main-text}
Suppose $X, Y$ are $(x,\gamma)$ and $(y,\gamma)$-SED respectively, then $X+Y$ is $(x+y,\gamma)$-SED and $XY$ is $(Nxy, \gamma)$-SED.
\end{lemma}
Lemma \ref{lemma:spatially-decay-M-H} already shows that $M_{km}, J_k$'s are $(c_M,\gamma_M)$-SED and $(c_J,\gamma_M)$-SED, respectively. Thus, by applying Lemma \ref{lemma:bound-AB-main-text} to the above recursive relation on $L^{(H),t}$, we can conclude that $L^{(H),t}_k$~'s are $(2(c_MNH+1)^{t-1}c_J, \gamma_M)$-SED by induction. However, note that this conclusion cannot be applied directly to show that $\LK$ is also SED, because although $L^{(H),t}$ preserves the SED property, the coefficient $2(c_MNH+1)^{t-1}c_J$ grows \emph{exponentially} w.r.t $t$, making the bound invalid if we let $t\to+\infty$.
\paragraph{Step 3: $L^{(H),t}$'s converge to $\LK$ exponentially fast} In the previous step, we cannot directly make $t\to+\infty$ as the SED coefficient blows up exponentially. One approach to fix this issue is that instead of letting $t$ go to infinity, we will stop at an appropriate threshold $t_0$ and show that $L^{(H),t_0}$ is a good enough approximation of $\LK$ already. This requires showing that $L^{(H),t}$ converges to $\LK$ in a relatively fast rate, so that we can pick a smaller threshold $t_0$. And this is indeed the case since
\begin{align*}
    &\|\LK\!\!-\!L^{(H),t}\| \!=\! \left\|\sum_{s=t}^{+\infty}\!\frac{{{\MK}'}^s \!\!\HK}{\lmax}\!\right\|\!\le\! \sum_{s=t}^{+\infty}\!\frac{\|{\MK}'\|^s \|\HK\|}{\lmax}\\
    &\le \frac{\|\HK\|}{\lmax}\sum_{s=t}^{+\infty}\left(1-\frac{\lmin}{\lmax}\right)^s\le \frac{\|\HK\|}{\lmin}e^{-\frac{\lmin}{\lmax}t},
\end{align*}
suggesting that the approximation error $\|\LK-L^{(H),t}\|$ goes to zero at an exponential rate.
\paragraph{Step 4: Combining the observations together} In this step, we seek to find an appropriate threshold $t_0$ mentioned in the previous step.
\begin{align*}
    &\quad \|[\LK_k]_{ij}\| \le  \|[\LK_k- L^{(H),t}_k]_{ij}\| + \|[L^{(H),t}_k]_{ij}\|\\
    &\le\!\!\!\underbrace{\frac{\|\HK\|}{\lmin}e^{-\frac{\lmin}{\lmax}t}}_{\textup{Exponential convergence to } \LK}\!\!+\underbrace{2(c_MNH+1)^{t-1}c_He^{-\gamma_M\dist(i,j)}}_{L^{(H),t}_k \textup{ is } (2(c_MNH+1)^{t-1}c_J, \gamma_M)\textup{-SED}} .
\end{align*}
Intuitively, we want the two terms to be roughly on the same scale. Here, we choose
\begin{equation*}
    t_0 = \left\lfloor\frac{\gamma_M\dist(i,j)}{\ln(c_MNH + 1)+\frac{\lmin}{\lmax}}\right\rfloor + 1.
\end{equation*}
Substituting $t_0$ into the above inequality, we have
\begin{align}
    &\quad \|[\LK_k]_{ij}\|\notag\\
    &\le \frac{\|\HK\|}{\lmin}e^{-\frac{\lmin}{\lmax}t_0}+2(c_MNH+1)^{t_0-1}c_He^{-\gamma_M\dist(i,j)}\notag \\
    & \le \left(\!\frac{\|\HK\|}{\lmin}\!+\!2c_J\!\right)\exp\left(\!{-\frac{\gamma_M\frac{\lmin}{\lmax}\dist(i,j)}{\ln(c_MNH \!+\! 1)\!+\!\frac{\lmin}{\lmax}}}\!\right)\label{eq:gamma_L}
\end{align}
which shows that $\LK_k$ is $\big(\frac{\|\HK\|}{\lmin}+2c_H,\frac{\gamma_M\frac{\lmin}{\lmax}}{\ln(c_MNH + 1)+\frac{\lmin}{\lmax}} \big)$-SED, thus completes the proof sketch.
\begin{remark}\label{rmk:decay-rate}
From the proof sketch, we can see that the SED rate $\gamma_K$ in \eqref{eq:gamma_L} depends on  two main factors: the first is the SED rate $\gamma_M$ for $M_{km}, H_k$, and the second is the condition number $\frac{\lmax}{\lmin}$ of $\MK$. Note that the SED rate $\gamma_M$ heavily rely on the SED rate of the system matrices $\gammasys$ itself. Further, it can be shown that the condition number of $\MK$ is smaller if $A$ is more stable, i.e., $\rho$ is larger (see Lemma \ref{lemma:bound-wM}). Thus we may conclude that $\gamma_K$ will be larger, i.e. the decaying rate is faster, if system matrices have larger decaying rate $\gammasys$, or the exponential stability factor $\rho$ is larger. \qed 
\end{remark}

\section{Numerical Simulation}\label{sec:numerics}
\subsection{A toy example}\label{sec:toy-example}
In this section we consider a toy example to confirm our results. The toy example shares the same $B,Q,R,S$ matrices as the counterexample defined in \eqref{eq:counter-example}, whereas $A$ is set as a stable matrix $A = e^{-\rho} I$. Clearly the above system satisfies both Assumption \ref{assump:exponential-decay} and \ref{assump:A-exponential-stable}. Figure \ref{fig:counter-example-stable-heatmap} plots the heatmap of the absolute values of the entries of $K$, visually demonstrating that entries of $K$ is spatially decaying, which forms a sharp contrast with the counter example (figure \ref{fig:counter-example-1}) in Section \ref{sec:preliminaries}. Figure \ref{fig:toy-example} demonstrates that $\|K_{ij}\|$ indeed decays exponentially with respect to $\dist(i,j)$, and that the decay rate is faster when $\rho$ is larger. This phenomenon can be explained by Remark \ref{rmk:decay-rate}, which concludes that the decay rate is faster if $\rho$ is larger.
\vspace{-5pt}
\begin{figure}[htbp]
    \centering
    \begin{minipage}{.47\linewidth}
    \includegraphics[width = .95\textwidth]{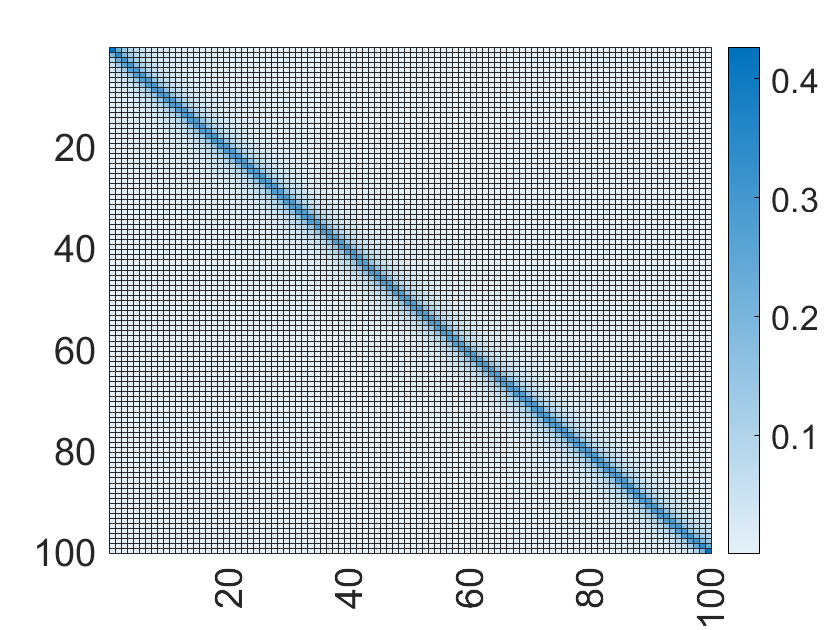}
    \caption{\small Heatmap of abs($K$), \\for $\rho = .1$}
    \label{fig:counter-example-stable-heatmap}
    \end{minipage}
    \begin{minipage}{.47\linewidth}
    \includegraphics[width = .95\textwidth]{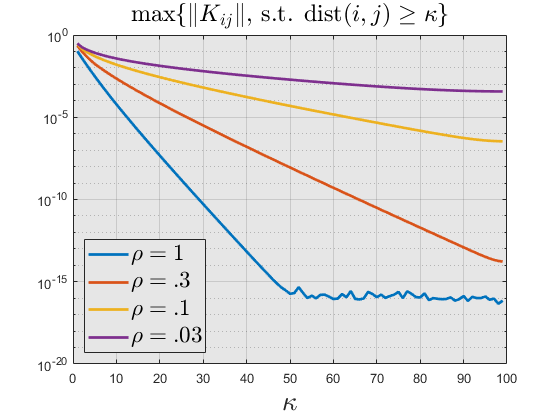}
    \caption{\small Numerical result for the toy example.}
    \label{fig:toy-example}
    \end{minipage}
\end{figure}
\vspace{-10pt}
\subsection{Practical linear systems}\label{subsec:numerical-practical}
\paragraph{Case Study 1 - Thermal Dynamics}
We adopt the linear thermal dynamic model used in \cite{zhang16decentralized} with additional process noise. Here it is assumed that all the constant disturbances (e.g., heat from external sources) are eliminated by expressing the state/control variables as deviation variables around the desired steady state; hence we obtain the following continuous time LQR problem for thermal dynamics
\begin{align*}
\textstyle \min_{\{u(t)\}}\int_0^{+\infty} \sum_{i=1}^N \alpha x_i(t)^2 + u_i(t)^2\qquad\\
\textstyle\dot{x_i} = \sum_{j\neq i, j\in \mathcal{N}_i} \frac{1}{v_i \zeta_{ij}} (x_j - x_i) + \frac{1}{v_i} u_i + \dot w_i.
\end{align*}
where $v_i$ is the thermal capacitance of zone $i$, $\zeta_{ij}$ represents the thermal resistance between two neighboring zones $i$ and $j$. Here we slightly abuse the notation and use $\dot w_i$ to denote the random disturbance for zone $i$.
The system parameters are set as $\zeta_{ij} = 1$\textcelsius/kW, $v_i = 200 + 20\times \mathcal{N}(0,1)$ kJ/\textcelsius, $\alpha = 3$, . We consider the underlying thermal interaction graph as a $10\times 10$ grid and discretize the continuous-time problem with $\Delta t = 1/4$ hour (detailed discretization scheme see Appendix \ref{apdx:why-SED} Eq\eqref{eq:discretization}).

\begin{figure}[htbp]
    \centering
    {\footnotesize Case Study 1: Thermal Dynamics}
    
    \includegraphics[width = .23\textwidth]{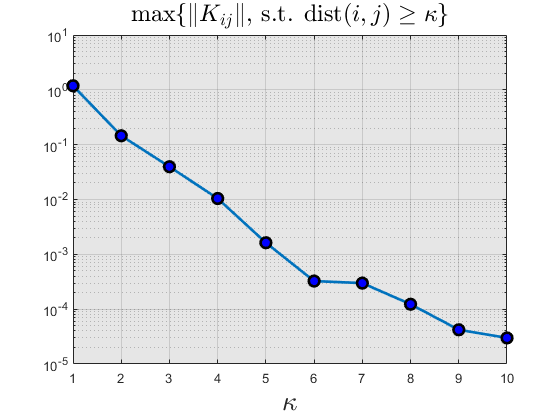}
    \includegraphics[width = .23\textwidth]{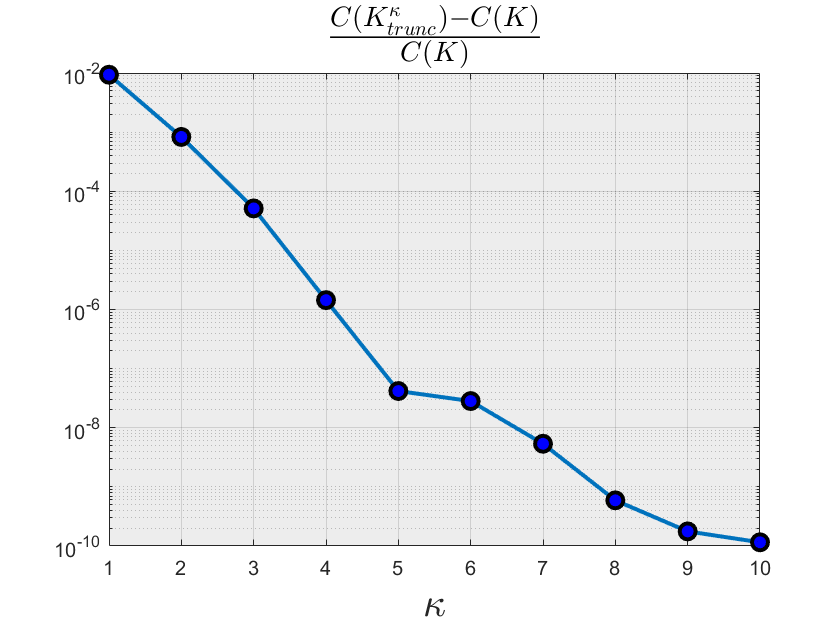}
    
    {\footnotesize Case Study 2: Frequency Control}
    
    \includegraphics[width = .23\textwidth]{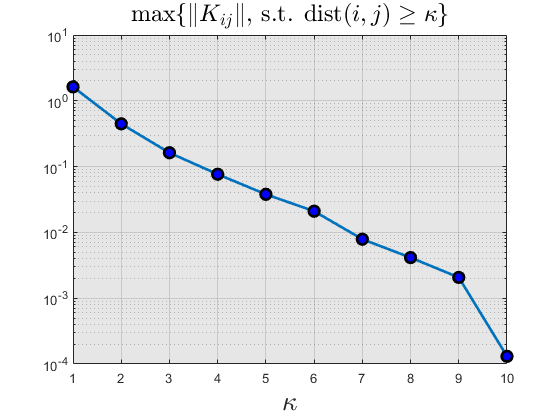}
    \includegraphics[width = .23\textwidth]{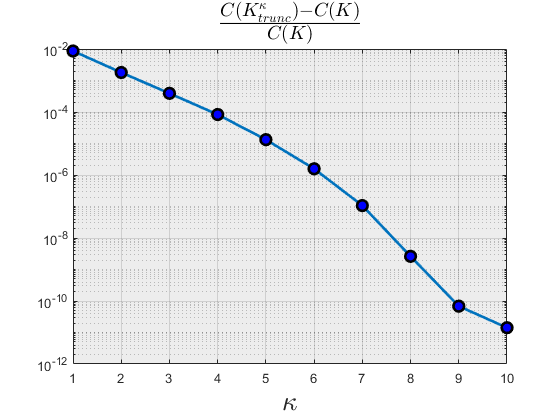}
    \caption{ Numerical results for Case Study 1 and 2. Figures on the left demonstrate how $\|K_{ij}\|$ decays w.r.t. $\kappa$ and figures on the right plot the performance difference ratio $\frac{C(\Ftrunc) - C(K)}{C(K)}$ for each example.}
    \label{fig:case-studies}\vspace{-15pt}
\end{figure}

\paragraph{Case Study 2 - Frequency control of power system}
We adopt the classical DM-approximated model for frequency control of power networks \cite{andreasson2014distributed}.
\begin{align*}
&\textstyle \min_{\{u(t)\}} \int_0^{+\infty} \sum_{i=1}^N \alpha_1 \theta_i(t)^2 + \alpha_2 \omega_i(t)^2 + u_i(t)^2\\
&\textstyle ~~~\dot{\theta_i} = \omega_i,~~\dot{\omega_i} =-\sum_{i\neq j,j\in\mathcal{N}_i}k_{ij}(\theta_j - \theta_i) + b_i u_i
\end{align*}
where $\theta_i,\omega_i$ are the phase angle and frequency of bus $i$ respectively. $k_{ij} = \frac{\ell_{ij}V_{\text{ref}^2}}{M}$, $\ell_{ij}$ is the line susceptance. Here we set $V_{\text{ref}} = 132$kV, $M = 10^5$kgm\textsuperscript{2}, $b_i = .1$, and $\alpha_1 = \alpha_2 = .5$. We use the power flow data for IEEE 145 bus to calculate the $\ell_{ij}$ as well as the graph structure. We discretize the continuous-time problem with $\Delta t \!=\! 5\!\times\! 10^{-6}$s. The decaying structure of $K$ is demonstrated in Figure \ref{fig:case-studies}, where we can observe that $\|K_{ij}\|$ indeed decays exponentially with respect to $\dist(i,j)$, and that the $\kappa$-truncated local  controller can already have near-optimal performance for relatively small $\kappa$.



\section{Conclusions and Future Directions}
This paper explores the spatially decaying structure for infinite-horizon discrete time LQR problem. We focus on the setting where the LQR problem is spatially exponential decaying (SED) and show that the optimal LQR state feedback gain $K$ is `quasi'-SED in this setting under certain stability conditions. Based on this result, we also analyze the near-optimal performance of $\kappa$-truncated local controllers in this setting. Additionally, as a side product of our proof, we also demonstrate that the optimal finite-truncated disturbance response controller also satisfies the `quasi'-SED property.

We believe that this is a fruitful direction with many interesting open questions, to name a few, whether we could improve the `quasi'-SED rate so that the $\textup{poly}\ln(N)$ term can be removed, how to handle infinite number of agents (i.e. $N\to +\infty$) and continuous time settings, and how to efficiently learn a $\kappa$-truncated local controller in the sample-based settings etc. In the long run, we hope to incorporate the results in this papers to more clever design of distributed learning and control of multi-agent systems. 



\bibliography{bib.bib}
\bibliographystyle{IEEEtran}

\appendix
\subsection{A Counter-example}\label{apdx:counter-example}

To show that the counter-example in Section \ref{sec:main-result} (Equation \ref{eq:counter-example}) does not contradict our theoretical result (Theorem \ref{theorem:spatial-decay-F} and Corollary \ref{coro:spatial-decay-F-unstable-A}), we now demonstrate that it violates Assumption \ref{assump:A-exponential-stable} and \ref{assump:stabilizability}. Since $A = 1.1I$, it clearly does not satisfy Assumption \ref{assump:A-exponential-stable}. The following lemma implies that it does not satisfy Assumption \ref{assump:stabilizability} either.
\begin{lemma}\label{lemma:counter-example}
For $A = cI$, $c\ge1$ and $B$ defined as in \eqref{eq:counter-example}, for any $\tau\ge 1,\rho >0$, there is no $\|K_0\|\le \frac{1-e^{-\rho}}{\tau} \sqrt{N}$ such that $A-BK_0$ is $(\tau,e^{-\rho})$-stable.
\end{lemma}
Define $\alpha \in \mathbb{R}^N$ as
\begin{align*}
    \alpha_i = \left\{
    \begin{array}{cc}
        1 & i \text{ is odd} \\
        -1 & i \text{ is even} 
    \end{array}
    \right..
\end{align*}
Define $\delta = B^\top \alpha$. We can quickly verify that $\|\delta\| = 1$ and $\|\alpha\| = \sqrt{N}$. For any $K_0$ such that $(A-BK_0)$ is $(\tau,e^{-\rho})$-stable:
\begin{align*}
    \tau e^{-k\rho}\|\alpha\| &\ge \|\alpha^\top (A-BK_0)^k\|\\
    &= \|c\alpha (A-BK_0)^{-k}-\alpha B K_0 (A-BK_0) ^{k-1}\|\\
    &\ge \|\alpha (A-BK_0)^{-k}\| - \|\delta\|\|K_0\|\tau e^{-\rho(k-1)}\\
    &\ge \dots\\
    &\ge \|\alpha\| - \|\delta\|\|K_0\|\sum_{t=0}^k\tau e^{-\rho t}\\
    &\ge \|\alpha \| - \frac{\tau}{1-e^{\rho}}\|\delta\|\|K_0\|\\
\Longrightarrow &\frac{\tau}{1-e^{\rho}}\|\delta\|\|K_0\|\ge \|\alpha \| - \tau e^{-k\rho}\|\alpha\|\\
\Longrightarrow &\|K_0\|\ge (1 - \tau e^{-k\rho})\frac{1-e^{-\rho}}{\tau}\sqrt{N}
\end{align*}
Set $k\to +\infty$ completes the proof.
Lemma \ref{lemma:counter-example} shows that any $K_0$ such that $A-BK_0$ is $(\tau, e^{-\rho})$-stable must be of norm $\gtrsim \sqrt{N}$, which leads to $\gamma_F \lesssim \frac{1}{N}$, making the rate invalid because the maximum distance for $\mathbb{Z}_N$ is $\lceil  N/2 \rceil $.

\subsection{Proof of Lemma \ref{lemma:optimal-L-K}}\label{apdx:optimal-LK}
We further define some variables that will be used in this section. Define the matrices $\Lambda_{km}\in \bR^{n_x\times n_x}$ as:
\begin{align}\label{eq:def-Lambda}
   \Lambda_{km}&:= \begin{cases}
         GA^{k-m}, &  k\ge m\\
         (A^{m-k})^\top G, & k<m 
    \end{cases} , \quad k,m\ge 1.
\end{align}
Define  ${\bf{A}}^{(H)} \in \mathbb{R}^{n_x\times H n_x}$ as:
\begin{align*}
    {\bf{A}}_k^{(H)} &:= 
    \begin{cases}
        [A^{k-1}, A^{k-2},\dots, A, I, 0,\dots, 0], & k< H \\
        \left[A^{k-1},A^{k-2},\dots,A^{k-K}\right], & k\ge H
    \end{cases}
\end{align*}
Then, define $\mathbf\Lambda^{(H)} \!\in\! \bR^{Hn_x\times Hn_x }, \XiK\!\in\! \bR^{Hn_x\times Hn_x }, \BK\in\bR^{Hn_x \times Hn_u},\RK\in\bR^{Hn_u\times Hn_u},\SK\in\bR^{Hn_u\times Hn_x}$ as
\begin{align}
   \mathbf\Lambda^{(H)} &\!:=
    \begin{bmatrix}
         \Lambda_{11}& \Lambda_{12}&\cdots &\Lambda_{1H} \\
         \Lambda_{21}& \Lambda_{22}&\cdots &\Lambda_{2H}\\
         \vdots&\vdots & &\\
         \Lambda_{H1} & \Lambda_{H2}&\cdots &\Lambda_{HH}
    \end{bmatrix},~~
    \XiK\!:= \begin{bmatrix}
     0\\
     \mathbf A_1^{(H)}\\
     \vdots\\
     \mathbf A_{H-1}^{(H)}
    \end{bmatrix}\label{eq:def-Lambda-K},\\
    \BK&:= 
    \begin{bmatrix}
         B& &\\
          & \ddots&\\
          & &B
    \end{bmatrix},\quad 
    \RK:= \begin{bmatrix}
         R& &\\
          & \ddots&\\
          & &R
    \end{bmatrix},\vspace{3pt}\\
    &\qquad\qquad \qquad  \SK:= \begin{bmatrix}
         S& &\\
          & \ddots&\\
          & &S
    \end{bmatrix}. \label{eq:def-BK-RK}
\end{align}

We first prove the following lemma which writes out $C(\LK)$ explicitly.
\begin{lemma}
The cost function defined in \eqref{eq:LQR-disturbance-feedback} satisfies
\begin{equation*}
\begin{split}
    C(\LK) &\!=\! \tr\!\left(\!
    [I, \LK{}^\top]\!\left(\!
    \begin{bmatrix}
        I & 0 \\
        0 & \BK
    \end{bmatrix}^{\!\top}
    \!\!\!\LambdaK\!
    \begin{bmatrix}
        I & 0 \\
        0 & \BK
    \end{bmatrix}
    \right.\right.\\
    &+ \begin{bmatrix}
        0 & 0 \\
        0 & \RK
    \end{bmatrix}
    + 
    \begin{bmatrix}
        0 & 0 \\
        0 & \SK
    \end{bmatrix}
    \mathbf{\Xi}^{(H+1)}
    \begin{bmatrix}
        I & 0 \\
        0 & {\BK}
    \end{bmatrix}\\
   & +     \begin{bmatrix}
        I & 0 \\
        0 & {\BK}^\top
    \end{bmatrix}
    \mathbf{\Xi}^{(H+1)^\top}
\begin{bmatrix}
        0 & 0 \\
        0 & {\SK}^\top
    \end{bmatrix}
    \bigg)
    \begin{bmatrix}
         I  \\
         \LK 
    \end{bmatrix}
    \bigg).
\end{split}
\end{equation*}
\end{lemma}
\begin{proof}
We denote the $z$-transform of $\{x_t\}, \{u_t\}, \{w_t\}$ as 
\begin{align*}
    X(z):=\sum_{t=0}^\infty z^{-t} x_t,~~ U(z):=\sum_{t=0}^\infty z^{-t} u_t, ~~W(z):=\sum_{t=0}^\infty z^{-t} w_t.
\end{align*}
Let $$\LK(z):= z^{-1}\LK_1 + z^{-2}\LK_2 +\dots + z^{-H}\LK_H.$$
Then we have 
\begin{align*}
    \begin{cases}
    U(z) = \LK(z)W(z)\\
    X(z) = z^{-1}\left(AX(z)+BU(z)+W(z)\right) 
    \end{cases}\\
    \Rightarrow\quad 
    (I-z^{-1}A)X(z) = z^{-1}(I+B\LK(z))W(z),
\end{align*}
which gives
\begin{align*}
    &\quad X(z)= z^{-1}(I-z^{-1}A)^{-1}(I+B\LK(z))W(z)\\
    &= z^{-1}(I+z^{-1}A+z^{-2}A^2 +\dots)\\
    &\qquad(I+z^{-1}B\LK_1+z^{-2}B\LK_2 +\dots z^{-H}B\LK_H)W(z)\\
    & = \left(z^{-1}I + z^{-2} (A + B\LK_1)\right. \\
    &\qquad \left.+ z^{-3} (A^2 + AB\LK_1 + B\LK_2) + \cdots\right)W(z)\\
    &= \left(z^{-1}T_1 + z^{-2}T_2 + \cdots+ z^{-k} T_k + \cdots\right)W(z),
\end{align*}
where (we can derive by algebraic calculation)
\begin{align*}
    T_k = \AK 
    \begin{bmatrix}
        I & 0 \\
        0 & \BK
    \end{bmatrix}
    \begin{bmatrix}
         I  \\
         \LK 
    \end{bmatrix}.
\end{align*}
Thus we have that
\begin{align*}
    &\quad \lim_{T\to+\infty} \frac{1}{T}\mathbb{E}\sum_{t=0}^{T-1}x_t^\top Qx_t = \lim_{T\to+\infty}\frac{1}{T}\mathbb{E}\sum_{t=0}^{T-1}\tr(x_tx_t^\top Q) \\
    & = \tr\left(\sum_{t=1}^{+\infty} T_tT_t^\top Q\right)\quad \textup{(Lemma \ref{lemma:tf-second-moment})}\\
    & = \tr\left(\sum_{t=1}^{+\infty} T_t^\top QT_t\right) \\
    &= \tr\left([I,\LK{}^\top]
    \begin{bmatrix}
        I & 0 \\
        0 & \BK
    \end{bmatrix}^\top\left(\sum_{t=1}^{+\infty}\mathbf{A}_t^{(H+1)}{}^\top Q\mathbf{A}_t^{(H+1)}\right)\right.\\
    &\qquad\qquad\qquad\qquad\qquad\qquad\qquad\quad\left.\quad 
    \begin{bmatrix}
        I & 0 \\
        0 & \BK
    \end{bmatrix}
    \begin{bmatrix}
         I  \\
         \LK 
    \end{bmatrix}\right)\\
    &= \tr\left(\![I,\LK{}^\top\!]\!
    \begin{bmatrix}
        I & \!\!0 \\
        0 & \!\!\BK
    \end{bmatrix}^\top\!\!\!\!\!\!\LambdaK\!
    \begin{bmatrix}
        I & 0 \\
        0 & \BK
    \end{bmatrix}\!\!
    \begin{bmatrix}
         I  \\
         \LK 
    \end{bmatrix}\!\right),
\end{align*}
where the last equation follows from Lemma \ref{lemma:matrix-M} in Appendix \ref{apdx:auxiliary}. Similarly
\begin{align*}
    &\quad \lim_{T\to+\infty} \frac{1}{T}\mathbb{E}\sum_{t=0}^{T-1}u_t^\top Sx_t = \lim_{T\to+\infty}\frac{1}{T}\mathbb{E}\sum_{t=0}^{T-1}\tr(x_tu_t^\top S) \\
    & = \tr\left(\sum_{t=1}^{K} T_t{\LK_t}^\top S\right)\quad \textup{(Lemma \ref{lemma:tf-second-moment})}\\
    & = \tr\left(\sum_{t=1}^{K} {\LK_t}^\top ST_t\right) \\
    &= \tr\left(\LK{}^\top \SK \begin{bmatrix}
         T_1\\
         \vdots\\
         T_H
    \end{bmatrix}\right) \\
    &= \tr\left(
    \LK{}^\top\SK
    \begin{bmatrix}
         \mathbf{A}_1^{(H+1)}\\
         \vdots\\
         \mathbf{A}_H^{(H+1)}
    \end{bmatrix}
    \begin{bmatrix}
        I & 0 \\
        0 & \BK
    \end{bmatrix}
    \begin{bmatrix}
         I  \\
         \LK 
    \end{bmatrix}\right)\\
    &= \tr\!\left(\![I,\LK{}^\top]\!
     \begin{bmatrix}
        0 & 0 \\
        0 & \SK
    \end{bmatrix}
    \mathbf{\Xi}^{(H+1)}\!
    \begin{bmatrix}
        I & 0 \\
        0 & {\BK}
    \end{bmatrix}\!
    \begin{bmatrix}
         I  \\
         \LK 
    \end{bmatrix}\!\right).
\end{align*}

Additionally
\begin{align*}
    &\quad \lim_{T\to+\infty} \frac{1}{T}\mathbb{E}\sum_{t=0}^{T-1}u_t^\top Ru_t = \lim_{T\to+\infty} \frac{1}{T}\mathbb{E}\sum_{t=0}^{T-1}\tr(u_tu_t^\top R) \\
    &= \tr\left(\sum_{t=1}^{K} \LK_t\LK_t{}^\top R\right)\quad \textup{(Lemma \ref{lemma:tf-second-moment})}\\
    & = \tr\left(\sum_{t=1}^{K} \LK_t{}^\top R\LK_t\right) = \tr\left(\LK{}^\top \RK~\LK\right).
\end{align*}
Combining the two equations together we have that
\begin{equation*}
\begin{split}
    C(\LK&) \!=\! \tr\!\left(\!
    [I, \!\LK{}^\top]\!\left(\!
    \begin{bmatrix}
        I & 0 \\
        0 & \BK
    \end{bmatrix}^\top
    \!\!\!\!\!\LambdaK\!
    \begin{bmatrix}
        I & 0 \\
        0 & \BK
    \end{bmatrix}\right.\right.\\
    &
    + \begin{bmatrix}
        0 & 0 \\
        0 & \RK
    \end{bmatrix}
    + 
    \begin{bmatrix}
        0 & 0 \\
        0 & \SK
    \end{bmatrix}
    \mathbf{\Xi}^{(H+1)}
    \begin{bmatrix}
        I & 0 \\
        0 & {\BK}
    \end{bmatrix}\\
    &\left.\left.+     \begin{bmatrix}
        I & 0 \\
        0 & {\BK}^\top
    \end{bmatrix}
    \mathbf{\Xi}^{(H+1)^\top}
\begin{bmatrix}
        0 & 0 \\
        0 & {\SK}^\top
    \end{bmatrix}
    \right)
    \begin{bmatrix}
         I  \\
         \LK 
    \end{bmatrix}
    \right).
\end{split}
\end{equation*}
\end{proof}

Lemma \ref{lemma:optimal-L-K} is a direct corollary of the above lemma.
\begin{proof}(of Lemma \ref{lemma:optimal-L-K})

\begin{align*}
    C(\LK&) \!=\! \tr\!\left(\!
    [I, \!\LK{}^\top]\!\left(\!
    \begin{bmatrix}
        I & 0 \\
        0 & \BK
    \end{bmatrix}^\top
    \!\!\!\!\!\LambdaK\!
    \begin{bmatrix}
        I & 0 \\
        0 & \BK
    \end{bmatrix}\right.\right.\\
    &
    + \begin{bmatrix}
        0 & 0 \\
        0 & \RK
    \end{bmatrix}
    + 
    \begin{bmatrix}
        0 & 0 \\
        0 & \SK
    \end{bmatrix}
    \mathbf{\Xi}^{(H+1)}
    \begin{bmatrix}
        I & 0 \\
        0 & {\BK}
    \end{bmatrix}\\
    &\left.\left.+     \begin{bmatrix}
        I & 0 \\
        0 & {\BK}^\top
    \end{bmatrix}
    \mathbf{\Xi}^{(H+1)^\top}
\begin{bmatrix}
        0 & 0 \\
        0 & {\SK}^\top
    \end{bmatrix}
    \right)
    \begin{bmatrix}
         I  \\
         \LK 
    \end{bmatrix}
    \right).\\
    &=\tr\left(
    [I, \LK{}^\top]
    \begin{bmatrix}
        \Lambda_{11} & \HK{}^\top \\
        \HK & \MK
    \end{bmatrix}
    \begin{bmatrix}
         I  \\
         \LK 
    \end{bmatrix}
    \right),
\end{align*}
where the last equation can be verified by the definition of $\LambdaK,\XiK,\MK,\HK$. The last equation immediately leads to the fact that the optimal $\LK$ should solve
\begin{equation*}
\MK\LK + \HK=0.
\end{equation*}
\end{proof}

\subsection{Proof of Lemma \ref{lemma:spatially-decay-M-H}}\label{apdx:spatially-decay-M-H}
We first state the following helper lemma:
\begin{lemma}\label{lemma:spatially-decaying-G}
$GA^{m}$ is $(\frac{\tau^2\|Q\|}{1-e^{-2\rho}} + 2q,\frac{\gammasys \rho}{(\rho+\ln(aN))})$-SED, for all $m\ge 0$. $SA^m$ is $(s+\tau\|S\|,\frac{\gammasys \rho}{(\rho+\ln(aN))})$-SED, for all $m\ge 0$.
\end{lemma}
\begin{proof}
We first prove $X:= GA^m$ is $(\frac{\tau^2\|Q\|}{1-e^{-2\rho}} + 2q,\frac{\gammasys \rho}{(\rho+\ln(aN))})$-SED. From the definition, we have
\begin{align*}
    X = \sum_{k=0}^{+\infty} (A^k)^\top Q A^{k+m}.
\end{align*}
Define the finite-time truncated version of $X$ as
\begin{equation*}
    X_t :=\sum_{k=0}^{t-1} (A^k)^\top Q A^{k+m}.
\end{equation*}
Firstly, we have from Assumption \ref{assump:A-exponential-stable} that
\begin{align*}
    \|X - X_t\| &= \|\sum_{k=t}^{+\infty} (A^k)^\top Q A^{k+m}\|\\
    &\le  \sum_{k=t}^{+\infty} \|A\|^k \|Q\| \|A\|^{k+m}\\
    &\le \|Q\|\sum_{k=t}^{+\infty} \tau^2e^{-(2k+m)\rho}\\
    &= \frac{\tau^2\|Q\|}{1-e^{-2\rho}} e^{-(2t+m)\rho}.
\end{align*}
Secondly, from Lemma \ref{lemma:bound-AB} we know that $(A^k)^\top Q A^{k+m}$ is $(N^{2k+m}qa^{2k+m},\gammasys )$-SED. Since $(aN)^2\ge2$ it holds that
$$\sum_{k=0}^{t-1} N^{2k+m}qa^{2k+m} = qa^mN^m\frac{(aN)^{2t}-1}{(aN)^2-1}\le 2q(aN)^{2(t-1)+m}\!\!,$$
we obtain $X_t$ is $(2q(aN)^{2(t\!-\!1)+m},\gammasys )$-SED.

Lastly, combining the previous results together we get
\begin{align*}
    \|[X]_{ij}\|&\le \|[X-X_t]_{ij}\| + \|[X_t]_{ij}\|\\
    &\le \|X-X_t\| + x_te^{-\gammasys \dist(i,j)}\\
    &\le \frac{\tau^2\|Q\|}{1-e^{-2\rho}} e^{-(2t+m)\rho} + 2q(aN)^{2(t-1)+m}e^{-\gammasys \dist(i,j)}.
\end{align*}
Intuitively, we want the two terms to be on the same scale, that is, 
$$e^{-(2t+m)\rho} = (aN)^{2(t-1)+m}e^{-\gammasys \dist(i,j)}.$$
Therefore, we set $$t = \left\lfloor\frac{\gammasys \dist(i,j)}{2(\rho+\ln(aN))}-\frac{m}{2}\right\rfloor + 1,$$
which gives
\begin{align*}
    \|[X]_{ij}\|&\le\frac{\tau^2\|Q\|}{1-e^{-2\rho}} e^{-(2t+m)\rho} + 2q(aN)^{2(t-1)+m}e^{-\gammasys \dist(i,j)}\\
    &\le \left(\frac{\tau^2\|Q\|}{1-e^{-2\rho}} + 2q\right) \exp\left({-\frac{\gammasys \rho}{(\rho+\ln(aN))}\dist(i,j)}\right)
\end{align*}
which completes the proof for $GA^m$. We now prove that $SA^m$ is $(s+\tau\|S\|,\frac{\gammasys \rho}{(\rho+\ln(aN))})$-SED. The technique is similar as the above proof. From Lemma \ref{lemma:bound-AB} we have that $SA^k$ is $(N^ksa^k,\gammasys )$-SED. Thus
\begin{align*}
    &\quad \|[SA^m]_{ij}\| \le \min\{\|[SA^m]_{ij}\|, \|SA^m\|\}\\
    &\le \min\{N^msa^me^{-\gammasys \dist(i,j)} ,\tau\|S\|e^{-\rho m}\}\\
    &\le (s+\tau\|S\|)\min\{(aN)^me^{-\gammasys \dist(i,j)}, e^{-\rho m}\}.
\end{align*}
For $m \ge \frac{\gammasys \dist(i,j)}{\rho + \ln(aN)}$, we have that
\begin{align*}
    \min\{(aN)^me^{-\gammasys \dist(i,j)}, e^{-\rho m}\} \le e^{-\rho m}\le e^{-\frac{\gammasys \rho}{\rho+\ln(aN)} \dist(i,j)};
\end{align*}
for $m\le \frac{\gammasys \dist(i,j)}{\rho + \ln(aN)}$,
\begin{align*}
    \min\{(aN)^me^{-\gammasys \dist(i,j)}, e^{-\rho m}\} \le (aN)^me^{-\gammasys \dist(i,j)}\\
    \le e^{-\frac{\gammasys \rho}{\rho+\ln(aN)} \dist(i,j)}.
\end{align*}
Thus
\begin{align*}
     &\quad \|[SA^m]_{ij}\| \le \min\{\|[SA^m]_{ij}\|, \|SA^m\|\}\\
    &\le (s+\tau\|S\|)\min\{(aN)^me^{-\gammasys \dist(i,j)}, e^{-\rho m}\}\\
    &\le (s+\tau\|S\|)e^{-\frac{\gammasys \rho}{\rho+\ln(aN)} \dist(i,j)},
\end{align*}
which completes the proof.
\end{proof}

Lemma \ref{lemma:spatially-decay-M-H} is a direct corollary of the above lemma.
\begin{proof}(of Lemma \ref{lemma:spatially-decay-M-H})
Since $$J_k = B^\top GA^k + SA^{k-1},$$
combining Lemma \ref{lemma:bound-AB} and Lemma \ref{lemma:spatially-decaying-G}, we know that $J_k$ is $\left(bN\left(\frac{\tau^2\|Q\|}{1-e^{-2\rho}} + 2q\right) + s +\tau\|S\|,\frac{\rho}{(\rho+\ln(aN))}\gammasys \right)$-SED. Similar statement also follows for $M_{km}$.
\end{proof}

\subsection{Proof of Theorem \ref{theorem:spatially-decay-L}}\label{apdx:spatially-decay-L}
The proof of Theorem \ref{theorem:spatially-decay-L} builds on the following lemma:
\begin{lemma}\label{lemma:main-lemma-spatially-decay-L}
If $\MK, \HK$ satisfies
\begin{enumerate}
    \item $M_{km}$ is $(c_M,\gamma_M)$-SED and $J_k$ is $(c_J,\gamma_M)$-SED, $\forall~ 1\le k,m\le K$,~~($c_M, c_J\ge 1$);
    \item $\lmin I\preceq \MK \preceq \lmax I$,~~ ($\lmax\ge 1$).
\end{enumerate}
Then for $\LK= -\left(\MK\right)^{-1}\HK,$ it satisfies that $\LK_k$ is $(c_L^{(H)}, \gamma_L^{(H)})$-SED, with 
\begin{equation*}
c_L^{(H)}\!\! =\! \frac{\|\HK\|}{\lmin}+2c_J, ~ \gamma_L^{(H)}\!\! =\! \frac{\lmin}{\lmax\ln(c_MNH + 1)\!+\!\lmin}\gamma_M.
\end{equation*}
\end{lemma}
\begin{proof}
Let
\begin{align*}
    {\MK}' &:= I - \frac{\MK}{\lmax},
\end{align*}
then $0\preceq {\MK}' \preceq (1-\frac{\lmin}{\lmax}) I$.
Notice that $\LK$ can be written as summations of polynomials of $\MK$ multiplied by $\HK$
\begin{align*}
    \LK &= -\left(\MK\right)^{-1}\HK \\
    &= -\left(\lmax I - \lmax {\MK}'\right)^{-1}\HK\\
    &=-\frac{1}{\lmax}\left(I-{\MK}'\right)^{-1}\HK\\
    &=-\frac{1}{\lmax}\sum_{s=0}^{+\infty}\left({\MK}'\right)^s \HK.
\end{align*}
Define the truncated summation as
\begin{align*}
    L^{(H),t} &:= -\frac{1}{\lmax}\sum_{s=0}^{t-1}{{\MK}'}^s \HK,
\end{align*}
then we have
\begin{align*}
    \|\LK-L^{(H),t}\| &= \left\|\frac{1}{\lmax}\sum_{s=t}^{+\infty}{{\MK}'}^s \HK\right\|\\
    &\le \frac{1}{\lmax}\sum_{s=t}^{+\infty}\|{\MK}'\|^s \|\HK\|\\
    &\le \frac{\|\HK\|}{\lmax}\sum_{s=t}^{+\infty}\left(1-\frac{\lmin}{\lmax}\right)^s\\
    &= \left(1-\frac{\lmin}{\lmax}\right)^t \frac{\|\HK\|}{\lmin}.
\end{align*}

On the other hand, 
we will show the claim that $L^{(H),t}_k$ is $(\ell_t, \gamma_M)$-SED for all $1\le k\le H$ by induction, where
\begin{equation*}
    \ell_t = (2(c_MNH+1)^{t-1}-1)c_J.
\end{equation*}
For $t=1$, $L^{(H),1} = \frac{1}{\lmax}\HK$, thus $L^{(H),1}_k$ is $(c_J,\gamma_M)$-SED, which satisfies the claim. Assume that the SED property holds for $t$, then at $t+1$,
\begin{align*}
    L^{(H),t+1} &= {\MK}' L^{(H),t} - \frac{\HK}{\lmax}\\
    & = \left(I-\frac{\MK}{\lmax}\right)L^{(H),t}- \frac{\HK}{\lmax}\\
    \Longrightarrow~~L^{(H),t+1}_k &= -\sum_{m=1}^H\frac{M_{km}}{\lmax}L^{(H),t}_m + L^{(H),t}_k - \frac{J_k}{\lmax}.
\end{align*}
From Lemma \ref{lemma:bound-AB}, we know that for each $m$, $\frac{M_{km}}{\lmax}L^{(H),t}_m$ is $(\frac{c_M N\ell_t}{\lmax},\gamma_M)$-SED, which can be further bounded as $(c_M N\ell_t,\gamma_M)$-SED from $\lmax\ge1$, thus $L^{(H),t+1}_k$ is $\big((c_MNH+1)\ell_t + c_J,\gamma_M\big)$-SED. From the induction assumption and $c_MNH\ge1$,
\begin{align*}
    &\quad (c_MNH+1)\ell_t + c_J \\
    &= (c_MNH+1)(2(c_MNH+1)^{t-1}-1)c_J + c_J \\
    &= (2(c_MNH+1)^{t}-c_MNH)c_J \\
    &\le (2(c_MNH+1)^{t} - 1)c_J = \ell_{t+1},
\end{align*}
which completes the proof of the claim by induction.

Now combining the above results together, we have that
\begin{align*}
    &\quad \|[\LK_k]_{ij}\| \le  \|[\LK_k- L^{(H),t}_k]_{ij}\| + \|[L^{(H),t}_k]_{ij}\|\\
    &\le \|\LK- L^{(H),t}\| + \ell_t e^{-\gamma_M\dist(i,j)}\\
    &\le \left(1-\frac{\lmin}{\lmax}\right)^t \frac{\|\HK\|}{\lmin} + 2(c_MNH+1)^{t-1}c_He^{-\gamma_M\dist(i,j)}\\
    &\le\frac{\|\HK\|}{\lmin}e^{-\frac{\lmin}{\lmax}t}+2(c_MNH+1)^{t-1}c_He^{-\gamma_M\dist(i,j)} .
\end{align*}
Intuitively, we want the two terms to be roughly on the same scale, here, we choose
\begin{equation*}
    t = \left\lfloor\frac{\gamma_M\dist(i,j)}{\ln(c_MNH + 1)+\frac{\lmin}{\lmax}}\right\rfloor + 1,
\end{equation*}
substitute this $t$ into the above inequality, we have
\begin{align*}
    &\|[\LK_k]_{ij}\|     \!\le\! \frac{\|\HK\|}{\lmin}e^{\!-\!\frac{\lmin}{\lmax}t}\!+\!2(c_MNH\!\!+\!\!1)^{t-1}c_Je^{\!-\!\gamma_M\dist(i,j)} \\
    &\le \frac{\|\HK\|}{\lmin}\exp\left({-\frac{\gamma_M\frac{\lmin}{\lmax}\dist(i,j)}{\ln(c_MNH + 1)+\frac{\lmin}{\lmax}}}\right) \\
    & +2c_J\exp\!\left(\!\ln(c_MNH \!\!+\!\! 1)\frac{\gamma_M\dist(i,j)}{\ln(c_MNH \!\!+\!\! 1)\!+\!\frac{\lmin}{\lmax}}\!\right)e^{\!-\!\gamma_M\dist(i,\!j)}\\
    & = \left(\frac{\|\HK\|}{\lmin}+2c_J\right)\exp\left({-\frac{\gamma_M\frac{\lmin}{\lmax}\dist(i,j)}{\ln(c_MNH + 1)+\frac{\lmin}{\lmax}}}\right)\\
    & = \!\left(\!\frac{\|\HK\|}{\lmin}\!+\!2c_J\!\right)\exp\left({\!-\frac{\gamma_M\lmin\dist(i,j)}{\lmax\ln(c_MNH + 1)+\lmin}}\!\right),
\end{align*}
which completes the proof.
\end{proof}
In order to apply Lemma \ref{lemma:main-lemma-spatially-decay-L} to show that $\LK$ is SED, we need to further give an upperbound and lowerbound of the eigenvalues of $\MK$, which is stated in the following lemma.
\begin{lemma}\label{lemma:bound-wM}
The eigenvalues of $\MK$ have bounds
\begin{align*}
    \lmin(\MK)&\ge \lmin(R - S Q^{-1}S^\top),\\
    \lmax(\MK)&\le \lmax(R) + \frac{4\tau^2(\|B\|^2\|Q\|+\|B\|\|S\|)}{(1-e^{-2\rho})^2}.
\end{align*}
\end{lemma}
\begin{proof}
From the definition of $\MK, \mathbf{\Lambda}^{(H)}, \XiK,\BK, \RK,\SK$ (c.f. \eqref{eq:def-MK-HK},\eqref{eq:def-Lambda-K},\eqref{eq:def-BK-RK}) we can verify that
\begin{align*}
    \MK = \BK{}^\top \mathbf{\Lambda}^{(H)} \BK + \SK\XiK\BK \\
    + \left(\SK\XiK\BK\right)^\top + \RK.
\end{align*}
Define $\QK \in \mathbb{R}^{Hn_x\times Hn_x}$ as
\begin{equation*}
    \QK = \left[\begin{array}{ccc}
         Q& & \\
         & \ddots&\\
         & & Q
    \end{array}\right],
\end{equation*}
then we have that
\begin{align*}
    &\MK = \BK{}^\top \mathbf{\Lambda}^{(H)} \BK + \SK\XiK\BK\\
    &\qquad\qquad\qquad\qquad+ \left(\SK\XiK\BK\right)^\top + \RK\\
    &= \BK{}^\top \mathbf{\Lambda}^{(H)} \BK \!+ \SK\XiK\BK \!\!+\! \left(\!\SK\XiK\BK\!\right)^\top \\
    &+ \SK \!\left(\QK\right)^{-1}\!\!\SK{}^\top\!\!+\! \left(\RK\!\! -\!\! \SK \left(\!\QK\!\right)^{-1}\SK{}^\top\right)\\
    &= \BK{}^\top \mathbf{\Lambda}^{(H)} \BK - \BK{}^\top \XiK{}^\top \QK\XiK \BK\\
    &+ \!\!\left(\!\QK\XiK\BK \!\!+\! \SK{}^\top\!\right)^{\!\!\!\top}\!\!{\!\QK\!}^{-\!1}\!\!\left(\!\QK\XiK\BK \!\!+\! \SK{}^\top\!\right)\\
    &+ \left(\RK - \SK \left(\QK\right)^{-1}\SK{}^\top\right)\\
    &\succeq \BK{}^\top\left( \mathbf{\Lambda}^{(H)}- \XiK{}^\top \QK\XiK\right) \BK \\
    &\quad + \left(\RK - \SK \left(\QK\right)^{-1}\SK{}^\top\right).
\end{align*}
Additionally from the definition of $\XiK$ (Eq \ref{eq:def-Lambda-K}) we have that
\begin{align*}
    \XiK{}^\top \QK\XiK = \sum_{t=1}^{H-1} \!\mathbf{A}_t^{(H)}{}^\top \!Q\mathbf{A}_t^{(H)} \\
    \preceq \sum_{t=1}^{\infty} \mathbf{A}_t^{(H)}{}^\top Q\mathbf{A}_t^{(H)} = \LambdaK,
\end{align*}
thus
\begin{align*}
    \MK\!\succeq \!\left(\!\RK \!\!-\! \SK \left(\!\QK\!\right)^{\!-\!1}\SK{}^\top\!\!\right) \!\succeq\!\lmin(R\!-\! SQ^{-1}S^\top)
\end{align*}
which proves the first inequality. For the second inequality, using Lemma \ref{lem:lmax} we have
\begin{align*}
    &\quad\lmax(\MK)\le \max_k \sum_{m=1}^H\left\|M_{km}\right\|\\
    &\le \!\|R\|\! +\! \|B^\top GB\| \!+\! 2\!\sum_{m=1}^{+\infty} \!\|B\|^2\|GA^m\| \!+\! 2\!\sum_{m=0}^\infty \!\|B\|\|S\|\|A^m\|\\
    &\overset{\text{(Lemma \ref{lemma:bound-G-norm})}}{\le} \|R\| \!+\! \frac{\tau^2\|B\|^2\|Q\|}{1-e^{-2\rho}} \big( 1\!+\! 2\sum_{m=1}^{+\infty}e^{-\rho m} \big) \!+\! \frac{2\tau\|B\|\|S\|}{1-e^{-\rho}}\\
    &\le \lmax(R) + \frac{4\tau^2(\|B\|^2\|Q\|+\|B\|\|S\|)}{(1-e^{-2\rho})^2}.
\end{align*}
\end{proof}

Theorem \ref{theorem:spatially-decay-L} is a direct corollary of Lemma \ref{lemma:spatially-decay-M-H}, \ref{lemma:main-lemma-spatially-decay-L} and \ref{lemma:bound-wM}. We first restate Theorem \ref{theorem:spatially-decay-L} in a more formal way.
f\begin{lemma}\label{lemma:spatially-decay-L-formal} (Theorem \ref{theorem:spatially-decay-L} restatement)
The optimal $\LK$ that solves \eqref{eq:L-tilde-eq} satisfies that $\LK_k$ is $(c_L^{(H)}, \gamma_L^{(H)})$-SED for any $1\le k\le H$, where 
\begin{align*}
    &\textstyle c_L^{(H)} = \frac{2\tau^2(\|B\|\|Q\|+\|S\|)}{(1-e^{-2\rho})^2\lmin(R\!-\!SQ^{-1}S^\top)} + 2c_J, \\
    &\textstyle \gamma_L^{(H)}  = \frac{\gammasys \rho\lmin(R\!-\!SQ^{-1}S^\top)}{\rho+\ln(aN)}\times\\
    &{\textstyle\frac{1}{\big(\lmax(R)+ \frac{4\tau^2(\|B\|^2\|Q\|+\|B\|\|S\|)}{(1-e^{-2\rho})^2}\big)\ln(c_MNH \!+\! 1) + \lmin(R-SQ^{-1}S^\top)}},
\end{align*}
where $c_J, c_M$ is defined as in Lemma \ref{lemma:spatially-decay-M-H}.
\end{lemma}

\begin{proof}
From Lemma \ref{lemma:spatially-decay-M-H} and Lemma \ref{lemma:bound-wM}, we know that $\MK$ and $\HK$ satisfies the conditions in Lemma \ref{lemma:main-lemma-spatially-decay-L} with
\begin{align*}
&c_M = b^2N^2\left(\frac{\tau^2\|Q\|}{1-e^{-2\rho}}+2q\right) + bN(s+\tau\|S\|)+ r,\\
&c_J = bN\left(\frac{\tau^2\|Q\|}{1-e^{-2\rho}} + 2q\right) + s + \tau\|S\|,\\
&\gamma_M = \frac{\rho\gammasys }{(\rho+\ln(aN))},\\
&\lmax = \lmax(R) \!+\! \frac{4\tau^2(\|B\|^2\|Q\|\!+\!\|B\|\|S\|)}{(1-e^{-2\rho})^2} ,\\
&\lmin =\lmin(R\!-\!SQ^{-1}S^\top) .
\end{align*}
Further
\begin{align*}
    \|\HK\| &\le \sum_{k=1}^H \|\HK_k\| \le \sum_{k=1}^H \|B\|\|GA^H\| + \|S\|\|A^{H-1}\|\\
    &\le \frac{\tau^2\|B\|\|Q\|}{1-e^{-2\rho}}\sum_{k=1}^H e^{-\rho k} + \|S\|\sum_{k=0}^{H-1}e^{-\rho k} ~~\textup{(Lemma \ref{lemma:bound-G-norm})}\\
    &\le \frac{2\tau^2(\|B\|\|Q\|+\|S\|)}{(1-e^{-2\rho})^2}.
\end{align*}
Substituting these quantities to Lemma \ref{lemma:main-lemma-spatially-decay-L} completes the proof, i.e.,
\begin{align*}
    c_L^{(H)} \!\!=  \!\frac{\|\HK\|}{\lmin}\!+\! 2c_J \!\le\!  \frac{2\tau^2(\|B\|\|Q\|+\|S\|)}{(1\!-\!e^{-2\rho})^2\lmin(R\!-\!SQ^{-1}S^\top)} \!+\! 2c_J.
\end{align*}
Further
\begin{align*}
    &\quad \textstyle\gamma_L^{(H)} = \frac{\gamma_M\lmin}{\lmax\ln(c_MNH + 1)+\lmin}\\
    &\textstyle = \frac{\gammasys \rho\lmin(R\!-\!SQ^{-1}S^\top)}{\rho+\ln(aN)}\times\\
    &{\textstyle\frac{1}{\big(\lmax(R)+ \frac{4\tau^2(\|B\|^2\|Q\|+\|B\|\|S\|)}{(1-e^{-2\rho})^2}\big)\ln(c_MNH \!+\! 1) + \lmin(R-SQ^{-1}S^\top)}},
\end{align*}
which proves Lemma \ref{lemma:spatially-decay-L-formal}.
\end{proof}

\subsection{Proof of Theorem \ref{theorem:spatial-decay-F} and Corollary \ref{coro:spatial-decay-F-unstable-A}}\label{apdx:proof-SED-F}
We are now ready to prove our main result on the quasi-SED of the optimal state feedback gain $K$. We first give a formal re-statement of Theorem \ref{theorem:spatial-decay-F}.
\begin{theorem}\label{theorem:spatial-decay-F-formal}(Theorem \ref{theorem:spatial-decay-F} formal statement)
The optimal control gain $K$ for problem \eqref{eq:LQR} is $(c_K,\gamma_K)$-SED, with
\begin{align*}
    &\textstyle c_K = \frac{2\tau^3(\|B\|^2\|K\|\|Q\|+\|B\|\|K\|\|S\|)}{(1-e^{-2\rho})^{5/2}\lmin(R\!-\!SQ^{-1}S^\top)} \\
    &\textstyle\quad+ \frac{2\tau^2(\|B\|\|Q\|+\|S\|)}{(1-e^{-2\rho})^2\lmin(R\!-\!SQ^{-1}S^\top)} + 2c_J\\
    &\textstyle \gamma_K= \frac{\gammasys \rho\lmin(R\!-\!SQ^{-1}S^\top)}{(\rho+\ln(aN))}\times \\
    &\textstyle\frac{1}{\big(\!\lmax\!(R) \!+\! \frac{4\tau^2(\|B\|^{\!2}\|Q\|\!+\!\|B\|\|S\|)}{(1-e^{-2\rho})^2}\!\big)\!\ln(\gammasys c_MN^2 \!+\! c_MN \!+\! 1) \!+\! \lmin\!(R\!-\!SQ^{\!-1}S^\top\!)},
\end{align*}
where $c_J, c_M$ is defined as in Lemma \ref{lemma:spatially-decay-M-H}.
\end{theorem}
\begin{proof}[Proof of Theorem \ref{theorem:spatial-decay-F-formal}]
\begin{align*}
    &\quad \|[K]_{ij}\| \le \|[K]_{ij}+[\LK_1]_{ij}\| + \|[\LK_1]_{ij}\|\\
    &\le \|K+\LK_1\| + \|[\LK_1]_{ij}\|\\
    &\le \frac{2\tau^3(\|B\|^2\|K\|\|Q\|\!+\!\|B\|\|K\|\|S\|)}{\lmin(R\!-\!SQ^{-1}S^\top)(1\!-\!e^{-2\rho})^{5/2}}e^{\!-\!H\rho} \!+\! c_L^{(H)}e^{\!-\!\gamma_L^{(H)}\dist(i,\!j)}
\end{align*}
Take
\begin{align*}
    H = \left\lfloor{\gammasys  N}\right\rfloor + 1.
\end{align*}
From the definition of $\gamma_L^{(H)}$, we have that
\begin{align*}
    H \ge {\frac{\gamma_L^{(H)} N}{\rho}}, ~~\forall~H\ge 1,
\end{align*}
thus
\begin{align*}
    &\quad \|[K]_{ij}\| \\
    &\le \frac{2\tau^3(\|B\|^2\|K\|\|Q\|\!+\!\|B\|\|K\|\|S\|)}{\lmin(R\!-\!SQ^{-1}S^\top)(1\!-\!e^{-2\rho})^{5/2}}e^{\!-\!\gamma_L^{\!(\!H\!)}\!N} \!\!+\! c_L^{(\!H\!)}e^{\!-\!\gamma_L^{(\!H\!)}\dist(i,\!j)}\\
    &\le \left(\!\frac{2\tau^3(\|B\|^2\|K\|\|Q\|+\|B\|\|K\|\|S\|)}{\lmin(R\!-\!SQ^{-1}S^\top)(1-e^{-2\rho})^{5/2}}\!+\! c_L^{(H)}\!\!\right)e^{\!-\gamma_L^{(H)}\dist(i,\!j)},
\end{align*}
i.e.,
\begin{align*}
    &\textstyle c_K = \frac{2\tau^3(\|B\|^2\|K\|\|Q\|+\|B\|\|K\|\|S\|)}{\lmin(R\!-\!SQ^{-1}S^\top)(1-e^{-2\rho})^{5/2}}+ c_L^{(H)} \\
    &\textstyle= \frac{2\tau^3(\|B\|^2\|K\|\|Q\|+\|B\|\|K\|\|S\|)}{(1-e^{-2\rho})^{5/2}\lmin(R\!-\!SQ^{-1}S^\top)} \\
    &\textstyle\qquad + \frac{2\tau^2(\|B\|\|Q\|+\|S\|)}{(1-e^{-2\rho})^2\lmin(R\!-\!SQ^{-1}S^\top)} + 2c_J\\
    &\textstyle \gamma_K = \gamma_L^{(H)}\\
    &\textstyle\ge \frac{\gammasys \rho\lmin(R\!-\!SQ^{-1}S^\top)}{(\rho+\ln(aN))}\times \\
    &\textstyle\frac{1}{\big(\!\lmax\!(R) \!+\! \frac{4\tau^2(\|B\|^{\!2}\|Q\|\!+\!\|B\|\|S\|)}{(1-e^{-2\rho})^2}\!\big)\!\ln(\gammasys c_MN^2 \!+\! c_MN \!+\! 1) \!+\! \lmin\!(R\!-\!SQ^{\!-1}S^\top\!)},
\end{align*}
thus completes the proof.
\end{proof}

\begin{proof}[Proof of Corollary \ref{coro:spatial-decay-F-unstable-A}]
We define $\bar u_t = K_0x_t + u_t$, then the LQR problem \ref{eq:LQR} can be re-expressed to the following problem
\begin{equation}\label{eq:LQR-unstable-re-express}
    \begin{split}
\min_{\{u_t\}_{t=0}^{\infty}}&\lim_{T\to\infty} \frac{1}{T}\mathbb{E}\sum_{t=0}^{T-1}x_t^\top \bar Q x_t + \bar u_t^\top \bar R\bar u_t + 2 \bar u_t^\top \bar S x_t\\
 s.t.~~ x_{t+1} &= \bar Ax_t + \bar Bu_t + w_t, ~~w_t\sim\mathcal{N}(0,I),    
 \end{split}
\end{equation}
where
\begin{align*}
    \bar A &= A-BK_0, ~~~~\bar B = B\\
    \bar Q & = Q - K_0^\top S - S^\top K_0 + K_0^\top R K_0\\
    \bar S &= S- K_0^\top R,~~~~\bar R = R.
\end{align*}
Since from Assumption \ref{assump:stabilizability}, $\bar A$ now satisfies the exponential stability assumption (Assumption \ref{assump:A-exponential-stable}). Additionally, we can show that $\bar A$ is $(a + bk_0N, \gammasys)$-SED, $\bar B$ is $(b,\gammasys)$-SED, $\bar Q$ is $(q+2sfN+k_0^2rN^2)$-SED, $\bar S$ is $(S-k_0rN)$-SED, $\bar R$ is $(r,\gammasys)$-SED. Thus Assumption \ref{assump:exponential-decay} is also satisfies. Thus, we could apply Theorem \ref{theorem:spatial-decay-F} to problem \eqref{eq:LQR-unstable-re-express} and get that the optimal controller is $\bar u_t = \bar K x_t$, where the optimal control gain $\bar K$ is $(c_K, \gamma_K)$-SED. Then the optimal control gain $K$ for the original problem \ref{eq:LQR} is $K = \bar K - K_0$, which is also SED.
\end{proof}

\subsection{Proof of Theorem \ref{theorem:performance-truncated}}\label{apdx:performance-truncated}
We first show that for $\kappa$ large enough, $\Ftrunc$ is a good approximation of $K$ both in terms of $\|\cdot\|$ and $\|\cdot\|_F$.
\begin{lemma}\label{lemma:bound F-Ftrunc}
\begin{align*}
    &\|K - \Ftrunc\| \le \sqrt{N} c_K e^{-\gamma_K \kappa},\\
    &\|K - \Ftrunc\|_F \le \sqrt{N\min\{n_x,n_u\}} c_K e^{-\gamma_K\kappa}
\end{align*}
\begin{proof}
From the definition of $\Ftrunc$ and that $K$ is $(c_K,\gamma_K)$-SED we have that $\|[K-\Ftrunc]_{ij}\|\le c_K e^{-\gamma\kappa}$. Thus for $x\in \bR^{n_x}$,
\begin{align*}
    &\hspace{-25pt}\left\|[(K\!-\!\Ftrunc) x]_i\right\|\! = \!\left\|\sum_{j}[K\!-\!\Ftrunc]_{ij}x_j\right\| \!\le\! c_Fe^{\!-\gamma_K\kappa} \!\sum_{j=1}^N\!\|x_j\|\\
    \Longrightarrow~~ &\left\|K-\Ftrunc x\right\|^2 \le \left(c_Fe^{-\gamma_K\kappa}\right)^2\left(\sum_{j=1}^N\|x_j\|\right)^2\\
    &\le N\left(c_Fe^{-\gamma_K\kappa}\right)^2\left(\sum_{j=1}^N\|x_j\|^2\right) = N\left(c_Fe^{-\gamma_K\kappa}\right)^2\|x\|^2\\
    \Longrightarrow~~ &\|K-\Ftrunc\|\le \sqrt{N}c_K e^{-\gamma_K\kappa}.
\end{align*}
Additionally, from the fact that $\|X\|_F \le \sqrt{\textup{rank}(X)}\|X\|$, where $\textup{rank}(X)$ is the rank of matrix $X$, we have that
\begin{equation*}
    \|K-\Ftrunc\|\le \sqrt{N\min\{n_x,n_u\}} c_K e^{-\gamma_K\kappa}
\end{equation*}
\end{proof}
\end{lemma}

Given the above lemma, we can directly call Lemma B.1 in \cite{tu2019gap} and get the following exponential stability result for $\Ftrunc$.
\begin{lemma}[Corollary of Lemma B.1 of \cite{tu2019gap}]\label{lemma:Ftrunc-stability}
For 
\begin{equation*}
    \kappa \ge \frac{\ln\left(\frac{2\tau c_K\sqrt{N}\|B\|}{1-e^{-\rho}}\right)}{\gamma_K},
\end{equation*}
$(A-B\Ftrunc)$~ is $ (\tau,\frac{1+e^{-\rho}}{2})$-stable.
\begin{proof}
From Lemma B.1 in \cite{tu2019gap}, we have that $A-B\Ftrunc$ satisfies $ \|(A-B\Ftrunc)^k\| \le \tau \left(\frac{1+e^{-\rho}}{2}\right)^k$ if
\begin{equation*}
     \|K-\Ftrunc\|\le \frac{1-e^{-\rho}}{2\tau \|B\|}.
\end{equation*}
Given Lemma \ref{lemma:bound F-Ftrunc}, the above condition is naturally satisfied if
\begin{align*}
    \sqrt{N} c_K e^{-\gamma_K \kappa} \le \frac{1-e^{-\rho}}{2\tau \|B\|}~\Longleftrightarrow~ \kappa \ge \frac{\ln\left(\frac{2\tau c_K\sqrt{N}\|B\|}{1-e^{-\rho}}\right)}{\gamma_K}
\end{align*}
\end{proof}
\end{lemma}
We are now almost ready to prove Theorem \ref{theorem:performance-truncated}. Before that, we first state an existing lemma that will be useful in the proof. We refer readers to the original papers for proof of the lemma.
\begin{lemma}(Lemma C.3 in \cite{krauth2019finite}, Lemmma 12 in \cite{fazel2018global}) \label{lemma:cost-difference-Ftrunc}
If $(A-BF')$ is $(\tau,e^{-\rho})$-stable, then
\begin{equation*}
    C(K') - C(K) \le \frac{\tau}{1-\rho^2}\|R + B^\top PB\|\|K' - K\|_F^2,
\end{equation*}
where $K$ is the optimal controller defined in \eqref{eq:F-expr} and $P$ is its cost-to-go matrix defined in \eqref{eq:ricatti-eq}.
\end{lemma}
\begin{proof}[Proof of Theorem \ref{theorem:performance-truncated}]
The proof of Theorem \ref{theorem:performance-truncated} is a combination of Lemma \ref{lemma:Ftrunc-stability} and \ref{lemma:cost-difference-Ftrunc}. From Lemma \ref{lemma:Ftrunc-stability}, we know that for $\kappa\ge \frac{\ln\left(\frac{2\tau c_K\sqrt{N}\|B\|}{1-e^{-\rho}}\right)}{\gamma_K}$, $ \|(A-B\Ftrunc)^k\| \le \tau \left(\frac{1+e^{-\rho}}{2}\right)^k$. Then, by applying Lemma \ref{lemma:cost-difference-Ftrunc},
\begin{align*}
    &\quad C(\Ftrunc) - C(K)\\
    &\le \frac{\tau}{1-\left(\frac{1+e^{-\rho}}{2}\right)^2} \|R+B^\top P B\|\|\Ftrunc - K\|_F^2\\
    &\le  \frac{\tau}{1-\left(\frac{1+e^{-\rho}}{2}\right)}\|R+B^\top P B\|\sqrt{N\min\{n_x,n_u\}} c_K e^{-\gamma_K\kappa}\\
    &= \frac{2\tau}{1-e^{-\rho}} \|R+B^\top P B\|\sqrt{N\min\{n_x,n_u\}} c_K e^{-\gamma_K\kappa},
\end{align*}
which completes the proof.
\end{proof}

\subsection{Auxiliary Lemmas}\label{apdx:auxiliary}
\begin{lemma}\label{lemma:bound-G-norm} 
For any $m\ge0$, G defined in \eqref{eq:def-G} satisfies
\begin{equation*}
    \|GA^m\| \le \frac{\tau^2\|Q\|e^{-\rho m}}{1-e^{-2\rho}}.
\end{equation*}
\end{lemma}
\begin{proof}
From the definition of $G$, we have
\begin{align*}
    \|GA^m\| &= \|\sum_{t=0}^\infty (A^t)^\top Q A^{t+m}\|\\
    &\le \sum_{t=0}^\infty \|Q\|\|A^t\|\|A^{t+m}\|\\
    &\le \|Q\|\sum_{t=0}^\infty \tau^2 e^{-\rho(2t+m)}\quad\text{(Assumption \ref{assump:A-exponential-stable})}\\
    &=\frac{\tau^2\|Q\|e^{-\rho m}}{1-e^{-2\rho}},
\end{align*}which completes the proof.
\end{proof}

\begin{lemma}\label{lemma:tf-second-moment}
Let $W(z)$ be the transfer function for $\{w_t\}$, defined as
\begin{equation*}
    W(z):= \sum_{t=0}^\infty z^{-t} w_t,~~w_t\sim \mathcal{N}(0,I),~i.i.d.
\end{equation*}
Let $T(z)$ be a stable transfer function matrix
$$T(z) = \sum_{t=0}^{+\infty} z^{-t} T_t.$$
Then for the  function $$X(z) := T(z)W(z)$$
which can be expanded as $\sum_{t=0}^\infty z^{-t} x_t,$
we have that
\begin{equation}\label{eq:lim-transfer}
    \lim_{t\to+\infty} \mathbb{E} x_tx_t^\top = \sum_{t=0}^\infty T_tT_t^\top.
\end{equation}
\end{lemma}
\begin{proof}
For any $t\ge0$, we have 
$$x_t = \sum_{t'=0}^t T_{t'} w_{t-t'}.$$
As $w_t$'s are i.i.d. standard Normal distributed, we have
$$\mathbb{E} x_tx_t^\top = \sum_{t'=0}^t T_{t'}T_{t'}^\top.$$
Letting $t$ goes to infinity and using the fact that $T$ is stable, we get \eqref{eq:lim-transfer}.
\end{proof}

\begin{lemma}\label{lem:lmax}
For any $H$-by-$H$ block matrix $X$ whose $(H,m)$-th block is $X_{km}$, $1\le k,m\le H$, we have 
\begin{equation}
    \lmax(X)\le\max_k\sum_{m=1}^H \|X_{km}\|.
\end{equation}
\end{lemma}
\begin{proof}
Any eigenvalue $\lambda$ and its corresponding eigenvector $v = (v_k)$ satisfies
$$Xv = \lambda v \quad\Rightarrow\quad \lambda v_k = \sum_{m=1}^H X_{km}v_m,~\forall 1\le k\le H.$$
Choosing $k^*$ such that $\|v_{k^*}\| = \max_m \|v_m\|$ gives that 
\begin{align*}
    &\quad |\lambda|\|v_{k^*}\|\le\sum_{m=1}^H\|X_{k^*m}\|\|v_m\|\le \sum_{m=1}^H\|X_{k^*m}\|\|v_{k^*}\|\\
    &\le \|v_{k^*}\| \max_k\sum_{m=1}^H \|X_{km}\|,
\end{align*}
which gives $|\lambda|\le \max_k\sum_{m=1}^H \|X_{km}\|$ for any eigenvalue $\lambda$. In specific, it holds for $\lambda=\lmax(X).$
\end{proof}

\begin{lemma}\label{lemma:matrix-M}
For any $H\ge0$, it holds that
\begin{equation}\label{eq:matrix-SK}
    \LambdaK = \sum_{t=1}^{+\infty}\mathbf{A}_t^{(H+1)}{}^\top Q\mathbf{A}_t^{(H+1)}.
\end{equation}
\end{lemma}
\begin{proof} 
Recall the definition of $\mathbf{A}_t^{(H+1)}$ is 
\begin{align*}
    \mathbf{A}_t^{(H+1)} &:= 
    \begin{cases}
        [A^{t-1}, A^{t-2},\dots, A, I, 0,\dots, 0], & k< K+1 \\
        \left[A^{t-1},A^{t-2},\dots,A^{t-K-1}\right], & k\ge K+1
    \end{cases}.
\end{align*}
Then for $1\le m\le k \le H+1$, the $(H,m)$-th block of the right-hand side of \eqref{eq:matrix-SK} is
\begin{align*}
    \sum_{t=0}^\infty (A^t)^{\!\top}\! Q A^{t\!+\!k\!-\!m} \!= \!\!\left(\!\sum_{t=0}^\infty (A^t)^\top\! Q A^{t}\!\right)A^{k\!-\!m} \!=\! GA^{k\!-\!m} \!=\! \Lambda_{km},
\end{align*}
where we used the definitions of $G$ and $\Lambda_{km}$ in \eqref{eq:def-G} and \eqref{eq:def-Lambda}, respectively. As for $k<m$, the symmetry gives the equality.
\end{proof}

\begin{lemma}\label{lemma:bound-AB}
Suppose $X\in \mathbb{R}^{n'\times n''}, Y\in \mathbb{R}^{n''\times n'''}$ (where $n', n'', n'''=n_x$ or $n_u$) are $(x,\gamma)$ and $(y,\gamma)$-SED, respectively. Then $XY$ is $(Nxy, \gamma)$-SED. Namely, if
\begin{align*}
    \|[X]_{ij}\|&\le xe^{-\gamma \dist(i,j)},\\ 
    \|[Y]_{ij}\|&\le ye^{-\gamma \dist(i,j)},
\end{align*}
then
\begin{equation*}
    \|[XY]_{ij}\| \le Nxye^{-\gamma \dist(i,j)}.
\end{equation*}
\end{lemma}
\begin{proof}
\begin{align*}
    \|[XY]_{ij}\| &= \|\sum_{k=1}^N [X]_{ik}[Y]_{kj}\| \le \sum_{k=1}^N \|[X]_{ik}\|\|[Y]_{kj}\|\\
    & \le \sum_{k=1}^N xye^{-\gamma_0 \left(\dist(i,k)+\dist(H,j)\right)}\\
    &\le \sum_{k=1}^N xye^{-\gamma_0 \dist(i,j)} = Nxy e^{-\gamma_0 \dist(i,j)},
\end{align*}
where the last inequality follows from the triangular inequality of the distance.
\end{proof}

\subsection{Why SED?}\label{apdx:why-SED}
 One important motivation of studying SED systems arises from practical implementation/discretization of the continuous time network LQR control considered in vast literature (e.g. \cite{Bamieh02}):
\begin{align*}
\min_{\{u(t)\}} &\int_0^{+\infty} x(t)^\top Q_cx(t) + u(t)^\top R_cu(t)\\
&\dot{x} = A_c x + B_c u,
\end{align*}
In many applications, such as thermal dynamics (see Example \ref{example:circle} below), $A_c, B_c$ are sparse matrices with $[A_c]_{ij}, [B_c]_{ij} = 0$ if $\dist(i,j) >1$, i.e., the corresponding matrices entries will be zero if $i,j$ are not neighbors of each other. 

In order to practically solve the continuous time control problem, discretization is needed. Standard result shows that
\begin{align*}
    x(t+\Delta t) = e^{\Delta t A_c}x(t) + \left(\int_0^{\Delta t}e^{(t+\Delta t - s)A_c}B_cu(s)ds\right),
\end{align*}
thus we could apply the following discretization scheme .
\begin{align}\label{eq:discretization}
    A \!=\! e^{\Delta t A_c}, ~~B\!=\! \left(\!\int_0^{\Delta t}\!\!\!e^{sA_c}ds\!\right) B_c, ~~Q \!=\! \Delta t Q_c,~~ R \!=\! \Delta t R_c
\end{align}
Note that this discretization scheme is more accurate compared with doing forward Euler scheme (c.f. \cite{ascher1998computer}). However, for discretization scheme \eqref{eq:discretization} the sparsity structure is not preserved for $A,B$. Fortunately, it can be shown that although $A,B$ are not sparse, they are SED with respect to the original graph distance (Proposition \ref{prop:discretize-SED} ). Thus, in order to handle this type of problems, we consider a broader setting of SED systems rather than sparse network systems in this paper.

\begin{prop}\label{prop:discretize-SED}
For matrices $A_c, B_c$ that are sparse, i.e., $[A_c]_{ij}, [B_c]_{ij} = 0$ if $\dist(i,j) >1$, and consider discretization scheme \ref{eq:discretization}, then $A$ is  $(e^{\Delta t\|A_c\|},-\ln(\Delta t \|A_c\|))$-SED. $B$ is $((\Delta t)^2 \|A_c\|\|B_c\|e^{\Delta t \|A_c\|},-\ln(\Delta t \|A_c\|))$-SED.
\begin{proof}
\begin{align*}
    \|[A]_{ij}\| &= \|[e^{\Delta t A_c}]_{ij}\|= \left\|\left[\sum_{k=0}^{+\infty} \frac{(\Delta t A_c)^k}{k!}\right]_{ij}\right\|\\
    & =  \left\|\left[\sum_{k\ge \dist(i,j)}\frac{(\Delta t A_c)^k}{k!}\right]\right\|\\
    &\le \sum_{k\ge \dist(i,j)} \frac{(\Delta t \|A_c\|)^k}{k!}\\
    &\le (\Delta t \|A_c\|)^{\dist(i,j)}\sum_{k\ge \dist(i,j)} \frac{(\Delta t \|A_c\|)^{k-\dist(i,j)}}{k!}\\
    &\le (\Delta t \|A_c\|)^{\dist(i,j)}\sum_{k\ge 0} \frac{(\Delta t \|A_c\|)^k}{k!}\\
    &= e^{\Delta t \|A_c\|} e^{\ln(\Delta t \|A_c\|)\dist(i,j)}
\end{align*}
\begin{align*}
    &\quad     \|[B]_{ij}\|= \left\|\left[\left(\!\int_0^{\Delta t}\!\!\!e^{sA_c}ds\!\right) B_c\right]_{ij}\right\|\\
    &= \left\|\left[\left(\!\int_0^{\Delta t}\sum_{k=0}^{+\infty}\frac{(sA_c)^kB_c}{k!}\!\right) \right]_{ij}\right\|\\
    &=\left\|\left[\left(\!\int_0^{\Delta t}\sum_{k\ge \dist(i,j) - 1}\frac{(sA_c)^kB_c}{k!}\!\right) \right]_{ij}\right\|\\
    &\le \int_0^{\Delta t}\sum_{k\ge \dist(i,j) - 1}\frac{(s\|A_c\|)^k\|B_c\|}{k!}\\
    &\le \Delta t \sum_{k\ge \dist(i,j) - 1}\frac{(\Delta t\|A_c\|)^k\|B_c\|}{k!}\\
    &\le \Delta t (\Delta t \|A_c\|)^{-\dist(i,j)+1}\hspace{-10pt}\sum_{k\ge \dist(i,j) - 1}\hspace{-12pt}\frac{(\Delta t\|A_c\|)^{k+1-\dist(i,j)}\|B_c\|}{k!}\\
    &\le  (\Delta t)^2 \|A_c\|\|B_c\|(\Delta t \|A_c\|)^{-\dist(i,j)+1}\sum_{k\ge 0}\frac{(\Delta t\|A_c\|)^k}{k!}\\
    &= (\Delta t)^2 \|A_c\|\|B_c\|e^{\Delta t \|A_c\|} e^{\ln(\Delta t \|A_c\|)\dist(i,j)},
\end{align*}
which completes the proof.
\end{proof}
\end{prop}

\end{document}